\documentclass[11pt]{article}
\usepackage{color}
\usepackage{enumerate, graphicx, amssymb, amsmath, amsthm}
\usepackage{inputenc}
\usepackage{mathrsfs,tikz}
\usepackage[margin=1in]{geometry}
\usepackage{accents}
\usepackage{mathtools}

\newcommand{\mat}[2][rrrrrrrrrrrrrrrrrrrrrrrrrrrrrrrrrrrrrrr]{\left[ \begin{array}{#1}#2 \\ \end{array} \right]}

\newcommand{\R}{\mathbb{R}}
\newcommand{\C}{\mathbb{C}}

\newcommand{\ve}{\varepsilon}
\newtheorem{theorem}{Theorem}[section]
\newtheorem{prop}[theorem]{Proposition}
\newtheorem{corollary}[theorem]{Corollary}
\newtheorem{lemma}[theorem]{Lemma}
 \theoremstyle{remark}
\newtheorem{remark}{Remark}[section]
\newtheorem{ex}{Example}[section]
 \theoremstyle{definition}
 \newtheorem{definition}{Definition}[section]

\numberwithin{equation}{section}

 \renewcommand{\>}{\rangle}
 
\newcommand{\CC}{\mathbb{C}} 
  \newcommand{\FF}{\mathbb{F}}
\newcommand{\GG}{\mathbb{G}} \newcommand{\HH}{\mathbb{H}} 
 \newcommand{\KK}{\mathbb{K}} 
\newcommand{\LL}{\mathbb{L}} \newcommand{\MM}{\mathbb{M}}
\newcommand{\NN}{\mathbb{N}}\newcommand{\RR}{\mathbb{R}}
\newcommand{\Sbb}{\mathbb{S}}

\newcommand{\al}{\alpha}

\newcommand{\vphi}{\varphi}
\newcommand{\la}{\lambda}
\newcommand{\de}{\delta}
\newcommand{\vep}{\varepsilon}
\newcommand{\ep}{\epsilon}
\newcommand{\ga}{\gamma}
\newcommand{\Ga}{\Gamma}
\newcommand{\ka}{\ka}
\newcommand{\Om}{\Omega}
\newcommand{\Si}{\Sigma}
\newcommand{\si}{\sigma}
\newcommand{\De}{\Delta}

\newcommand{\vth}{\vartheta}
\newcommand{\om}{\omega}

\newcommand{\Ac}{{\mathcal{A}}}
\newcommand{\Ic}{{\mathcal{I}}}

\newcommand{\Mc}{{\mathcal{M}}}

\newcommand{\Mf}{\mathfrak{M}}

\newcommand{\mf}{\mathfrak{m}}

\newcommand{\Hr}{{\mathrm{H}}}
\newcommand{\nr}{{\mathrm{n}}}

\usepackage{bbm}
\newcommand{\one}{\mathbbm{1}}

\newcommand{\Bbf}{\mathbf{B}}
\newcommand{\Dbf}{\mathbf{D}}
\newcommand{\Ebf}{\mathbf{E}}
\newcommand{\fbf}{\mathbf{f}}

\newcommand{\Hbf}{\mathbf{H}}

\newcommand{\ubf}{\mathbf{u}}
\newcommand{\vbf}{\mathbf{v}}
\newcommand{\wbf}{\mathbf{w}}

\newcommand{\ybf}{\mathbf{y}}
\newcommand{\zbf}{\mathbf{z}}

\newcommand{\ii}{\mathrm{i}}
\newcommand{\ee}{\mathrm{e}}

\newcommand{\wt}{\widetilde}
\newcommand{\wh}{\widehat}
\newcommand{\pa}{\partial}
\newcommand{\uph}{\upharpoonright}

\DeclareMathOperator{\im}{Im}
\DeclareMathOperator{\re}{Re}

\DeclareMathOperator{\ran}{ran}

\DeclareMathOperator{\curl}{curl}
\DeclareMathOperator{\Div}{div}

\DeclareMathOperator{\Tr}{{Tr}}
\newcommand{\sym}{\mathrm{sym}}

\newcommand{\Right}{{\mathrm{right}}}

\newcommand{\imp}{\mathrm{imp}}

\newcommand{\E}{\mathbf{E}}
\renewcommand{\H}{\mathbf{H}}

\newcommand{\imb}{\hookrightarrow}
\newcommand{\<}{\langle}
\newcommand{\sa}{\mathrm{sa}}
\newcommand{\nsa}{\mathrm{nsa}}

\newcommand{\disc}{\mathrm{disc}}
\newcommand{\ess}{\mathrm{ess}}

\renewcommand{\Re}{{\mathrm{Re\;}}}
\renewcommand{\Im}{{\mathrm{Im\;}}}
 \renewcommand{\>}{\rangle}
\DeclareMathOperator{\Dr}{\Ga}
\newcommand{\Mo}{{M}}
\newcommand{\Mmin}{M_{0,0}}
\newcommand{\paOm}{{\pa \Om}}
\newcommand{\LLt}{{L^2_t (\pa \Om)}}

\DeclareMathOperator{\curlm}{\mathbf{curl}}
\DeclareMathOperator{\grad}{\mathbf{grad}}

\newcommand{\Himp}{{H_\imp  (\curlm, \Om)}}
\newcommand{\Le}{{\LL^2_{\ve,\mu} (\Om)}}
\newcommand{\LLOm}{{L^2 (\Om,\CC^3)}}
\newcommand{\Sem}{{\Sbb_{\ve,\mu} (\Omega)}}
\newcommand{\Tan}{\mathrm{tan}}
\newcommand{\Mabsym}{{L^\infty (\Om,\MM^\sym_{\al,\beta})}}
\newcommand{\Mab}{{L^\infty (\Om,\MM_{\al,\beta})}}
\newcommand{\Weak}{{\,\rightharpoonup\,}}

\newcommand{\Hcon}{{\,\xrightarrow{\, \mathrm{H} \, }\,}}
\newcommand{\rhoH}{{\rho_{\mathrm{H}}}}
\newcommand{\wrhoH}{{\wh \rho_{\mathrm{H}}}}

\DeclareMathOperator{\wlim}{\mathrm{w}-\!\lim}

\newcommand{\n}{\mathbf{n}}

\begin{document}
\title{Homogenization and nonselfadjoint spectral optimization for dissipative Maxwell eigenproblems} 
\author{}
\date{}
\maketitle
\thispagestyle{empty}

\vspace{-8ex}
{ \center{ \large
Matthias Eller$^\text{ a}$ and Illya M. Karabash$^\text{ b,*}$\\[3ex]
  }
}  
  
  {\small \noindent
$^{\text{a}}$ Department of Mathematics and Statistics, Georgetown University, Washington, DC 20057, USA \\[1mm]  
$^{\text{b}}$ Institute for Applied Mathematics, University of Bonn, Endenicher Allee 60,
53115 Bonn, Germany\\[0.2ex]
$^{\text{*}}$ Corresponding author: karabash@iam.uni-bonn.de\\[.2ex]
E-mails: mme4@georgetown.edu, karabash@iam.uni-bonn.de
}

\begin{abstract}
The homogenization of eigenvalues of non-Hermitian Maxwell operators is studied by the H-convergence method. It is assumed that the Maxwell systems are equipped with suitable m-dissipative boundary conditions, namely, with Leontovich or generalized impedance boundary conditions of the form $\n   \times \Ebf   = Z  [(\n  \times \Hbf )\times \n ] $. We show that, for a wide class of impedance operators $Z$, the nonzero spectrum of the corresponding Maxwell operator is discrete. To this end, a new continuous embedding theorem for domains of Maxwell operators is obtained. We prove the convergence of eigenvalues to an eigenvalue of a homogenized Maxwell operator under the assumption of the H-convergence of the material tensor-fields. This result is used then to prove the existence of optimizers for eigenvalue optimization problems and the existence of an eigenvalue-free region around zero. As applications, connections with the quantum optics problem of the design of high-Q resonators are discussed, and a new way of the quantification of the unique (and nonunique) continuation property   is suggested.
\end{abstract}

\noindent
MSC-classes: 
78M40  
49R05  
35Q61 
49J20   
47B44 
35J56   
\\[0.5ex]

\tableofcontents

\section{Introduction} \label{s:I}
This paper is aimed to  homogenization, optimization, and unique continuation for (generally non-selfadjoint) eigenproblems associated with the time-harmonic Maxwell system 
  \begin{equation}\label{e:EP}
     \ii  \nabla \times \Hbf (x) = \omega \ve (x) \Ebf (x), \qquad -\ii  \nabla \times \Ebf (x) = \omega \mu (x) \Hbf (x), \qquad x \in \Om, 
    \end{equation}
in a fixed Lipschitz domain $\Om \subset \RR^3$. The eigenfields $\{\Ebf,\Hbf\}$ of \eqref{e:EP}  are assumed to satisfy a fixed m-dissipative boundary condition on the boundary $\pa \Om$ of $\Om$, i.e.,  $\Ebf$ and $\Hbf$ satisfy a boundary condition that together with the  differential operation  $\Mf (\ve,\mu): \{\Ebf,\Hbf\} \mapsto \{\ii \ve^{-1} \nabla\times \Hbf, -\ii \mu^{-1} \nabla \times \Ebf\}$ defines an m-dissipative operator $\Mo$ in a Hilbert space equipped with the suitable  'energy norm'.

 The material parameters  $\ve$ and $\mu $, which are subjected to optimization in Section \ref{ss:OptRes}, are represented by two $L^\infty (\Om,\RR^{3\times 3})$-matrix-functions taking almost everywhere (a.e.) uniformly positive symmetric values.
 The Hilbert space
\begin{align}
& \text{$ \Le$ coincides with the orthogonal sum $L^2 (\Om,\CC^3)^2 = L^2 (\Om,\CC^3) \oplus L^2 (\Om,\CC^3) $}  \label{e:Le0}  
\\
&  \text{as a vector space, but is equipped with the equivalent norm  $ \| \cdot \|_\Le$ (energy norm)}
\label{e:Le1}\\
& \qquad  \qquad \text{  defined by \qquad
$
 \| \{\Ebf,\Hbf\} \|_\Le^2  =  ( \ve \Ebf , \Ebf)_\Om^2 + 
 ( \mu \Hbf, \Hbf)_\Om^2 ,$} \label{e:Le2}
\end{align}  
where $ (\cdot, \cdot)_\Om$ is the standard inner product of the complex Hilbert space $L^2 (\Om,\CC^3)$ of $\CC^3$-vector-fields. 
The other two equations of the Maxwell system,
Gauss's law and the absence of magnetic monopoles, 
can be taken into account by the consideration of the part 
$\Mo |_\Sem$ of the Maxwell operator $\Mo$ in the closed subspace \cite{C96,M03,ACL18}
\begin{equation} \label{e:Sem}
\Sem :=  \{ \{\Ebf,\Hbf\}  \in \Le \ : \ \Div (\ve \Ebf) = 0 \ , \ \Div (\mu \Hbf) = 0\}  .
\end{equation}

Our primary motivation is the Applied Physics problem of design of optical resonators with high
quality-factors $Q = - \frac{|\re \om|}{2 \im \om}$  \cite{LSV03,VS21}
and the associated mathematical problems of modelling \cite{M03,LL04,ACL18,EK22} and optimization \cite{HS86,KS08,HBKW08,K13,OW13,KKV20,EK21} of resonances in open optical, acoustic, or quantum cavities. The key feature of these problems is  the leakage of the energy to the outer medium  modelled by $\RR^3 \setminus \Om$. For the associated nonselfadjoint  Maxwell operators $\Mo$, this means mathematically that $(-\ii)\Mo$ is a  generator of a contraction semigroup $\{\ee^{-\ii \Mo t}\}_{t \ge 0}$ in the energy space $\Le$, or equivalently \cite{P59,Kato,E12}, that $\Mo$ is m-dissipative.

Despite numerous related Physics studies, the spectral optimization for nonselfadjoint differential operators is presently in the initial stage of its mathematical development. Relevant analytic results on the existence of optimizers have been obtained mainly in the cases where the weak-*-compactness of coefficients is applicable \cite{HS86,CZ95,K13,OW13,KKV20}.  For Maxwell systems, this approach is difficult to use since the corresponding homogenization convergences for $\vep (\cdot)$ and $\mu (\cdot)$  are the G-convergence of De Giorgi \& Spagnolo \cite{MT78}  and, equivalently (since $\vep $ and $\mu $ are symmetric), the H-convergence of Murat \& Tartar (see the monographs \cite{D93,C00,A02,T09}).  

Structural optimization problems for  eigenvalues of selfadjoint partial differential operators have been actively studied during  several decades  \cite{ATL89,CL96,H06,EL21,LZ21,GHL21,KCK20}. In particular, the H-convergence approach to the existence of optimizers 
was developed in \cite{CL96,A02}. Concerning the closely connected, but somewhat different theory of shape optimization, which stems from  the Rayleigh–Faber–Krahn inequality, we refer to  \cite{ATL89,A02,H06} and to the recent collection of works \cite{H17}.

An application of the homogenization method to the proof of the existence of optimal eigenvalues of 3-dimensional m-dissipative Maxwell operators would be a natural step forward. However, for the nonselfadjoint case, the H-convergence approach  meets the following difficulties:
\begin{itemize}
\item[(A)] The homogenization method, which was initially developed  for selfadjoint 2nd order elliptic operators with discrete spectrum (see \cite{CL96,A02}),  substantially uses the spectral theorem for selfadjoint operators and the Courant–Fischer–Weyl min-max principle. These  tools are not available (at least directly) in the case of  nonselfadjoint  operators.
 An alternative approach to the convergence of eigenvalues in the case of smooth periodic coefficients was used by Oleinik, Shamaev, and Yosifian \cite{OSY92} for (generally nonselfadjoint) higher order  elliptic operators
 $\frac 1\rho \sum\limits_{|\al|, |\beta| \le m } (-1)^{|\al|} \pa^\al a_{\al\beta} \pa^\beta$ 
acting in the space $L^2 (\Om)$ of scalar-fields. This approach relies on the special `compact convergence' for inverse operators and has its own specific difficulties. Namely, in addition to the compactness of  inverse operators, the assumption that the limiting eigenvalue is semi-simple is needed for the eigenvector convergence estimates. 
Note that the resolvent of  an arbitrary m-dissipative Maxwell operator $M$ in $\Le$ is obviously not compact since $M$ has an infinite dimensional kernel $\ker M = \{\Phi = \{\Ebf,\Hbf\} \in D(M) : M \Phi = 0\}$ (for the basic theory of Maxwell operators see, e.g., \cite{C96,M03,ACL18,EK22} and  Section \ref{ss:Disc}). 
\item[(B)] The question whether the restricted Maxwell operator $\Mo |_\Sem$ has compact resolvent is more involved since the compactness of the resolvent $(T-\om)^{-1}$ for a partial differential operator $T$ depends generally on the associated boundary conditions even in the case of a non-empty resolvent set $\rho (T) = \CC \setminus \si (T)$ \cite{GG91}. This means  that there is no guaranty  that the nonzero spectrum $\si (\Mo) \setminus \{0\}$ for an m-dissipative Maxwell operator $\Mo$ is discrete (i.e., consists of isolated eigenvalues of finite algebraic multiplicity). The description of the family of  boundary conditions such that  corresponding restricted m-dissipative Maxwell operators $\Mo |_\Sem$ have purely discrete spectra $\si (\Mo |_\Sem) = \si_\disc (\Mo |_\Sem) $ seems to be an open problem (cf. \cite{GG91} for the discussion of related questions for Laplacians in bounded domains).
\end{itemize}

The goal of this paper is to overcome the aforementioned difficulties for  Maxwell operators  associated with a sufficiently wide class of m-dissipative boundary conditions.

We first address the difficulty (B) and describe a class of generalized impedance boundary conditions that, for the restricted Maxwell operator  $\Mo |_\Sem$, ensure simultaneously  the discreteness of  spectrum and the  m-dissipativity, see Sections \ref{ss:Disc} and \ref{s:ProofsDisc}. 

The m-dissipative boundary conditions that we consider are generalizations of Leontovich  absorbing boundary conditions. In the Photonics settings  the medium in the outer region $\RR^3 \setminus \Om$ is not homogeneous and usually has a sophisticated structure, which may be partially unknown or random 
\cite{VS21}. Since it is difficult to justify the use of Silver-Müller radiation conditions at $\infty$ in this context, a variety of approximate absorption boundary conditions are employed in computational electromagnetism to model resonances of leaky optical cavities   \cite{LL84,M03,LL04,ACL18,YI18} (see also the discussion in \cite{EK22}). 
In particular, the Leontovich 
boundary  condition 
\begin{gather} \label{e:aBC}
   \n (x) \times \Ebf (x)  = a(x) [(\n (x) \times \Hbf (x))\times \n (x)] , \qquad x \in \pa \Om ,
\end{gather}
which was introduced in Radiophysics publications of Shchukin and Leontovich 
(see \cite{LL84,YI18}) and which is called also the impedance boundary condition, is widely used \cite{M03,LL04,ACL18}.

Here  $\n (\cdot)$ is  the exterior normal unit vector-field  along $\partial\Omega$, and   
the  scalar-valued measurable function $a: \pa \Om \to \overline\CC_\Right : = \{z \in \CC \ : \ \Re z \ge 0\}$ in \eqref{e:aBC} is called an \emph{impedance coefficient}. 
In this paper we are interested in impedance coefficients $a(\cdot)$ 
  satisfying the following set of assumptions:
\begin{align} \label{e:aLinf}
 a \in L^\infty (\pa \Om), 
\quad |a(x)|> c >0 \text{ a.e. } & \text{on $\pa \Om$ for a certain constant $c>0$},
\\
 \text{there exists $\de>0$ such that either } & \text{ $\arg a (x) \in [-\pi/2,\pi/2-\de]$ a.e. on  $ \pa \Om$,} \notag \\
 & \text{\qquad \qquad or $\arg a(x) \in [-\pi/2+\de,\pi/2]$ a.e. on $\pa \Om$,} \label{e:aCo}
\end{align}
where the inequalities and inclusions are valid a.e. with respect to (w.r.t.) the surface measure of $\pa \Om$, and 
$\arg z$ for $z \in \CC \setminus \{0\}$ is the complex argument of $z$, i.e., $\arg z \in (-\pi,\pi]$ and $z = |z| \ee^{\ii \arg z}$.

Following \cite{LL04}, we understand the boundary condition \eqref{e:aBC} in the sense that \linebreak $[(\n  \times \Hbf )\times \n] \in \LLt$,  \quad $\n   \times \Ebf \in \LLt $,  and equality \eqref{e:aBC} holds 
\[
\text{in  the Hilbert space } \quad \LLt = \{ \vbf \in L^2 (\pa \Om,\CC^3) : \ \n \cdot \vbf = 0 \text{ a.e. on } \pa \Om\} \ \text{}
\]
of  tangential $L^2$-vector-fields.
We denote the associated \emph{Maxwell-Leontovich operator} in $\Le$ by $M_{\imp,a}$.  

The operator 
$M_{\imp,a}$ is m-dissipative in $\Le$ if \eqref{e:aLinf}-\eqref{e:aCo} hold true \cite{EK22}
(for the case of uniformly positive $a (\cdot)$, see \cite{LL04}).
Note that, if $a (\cdot)$ vanishes on a set of positive surface measure or is unbounded, the aforementioned interpretation of \eqref{e:aBC} does not necessarily lead to an m-dissipative  operator \cite{EK22} (in the cases of degenerate or singular impedance coefficients $a$, other analytic interpretations of \eqref{e:aBC} are needed, see \cite{M03,ACL18,EK22,PS22} and references therein).

In the case of nonselfadjoint boundary conditions,  the spectral properties of Maxwell operators $M$ and, in particular, the discreteness of the nonzero spectrum $ \si (M) \setminus \{0\}$ are not adequately studied. The pioneering analytic result of S.G.~Krein \& Kulikov \cite{KK69} states that $\si (M) \setminus \{0\} = \si_\disc (M)$ in the particular case where $a \equiv c$ with a constant $c \in \overline\CC_\Right$ and $\pa \Om$ is smooth. The recent paper \cite{PS22} considers the case of mixed boundary condition with a constant impedance coefficient $a \equiv c >0$ for a Leontovich boundary condition on an `absorbing part' $\Ga^1 $ of the boundary $\pa \Om$, while the remaining part of the boundary $\Ga^0 = \pa \Om \setminus \overline{\Ga^1}$ is equipped with the 'perfect metal condition' $\n \times \Ebf = 0$. 
We were not able to find results on the the discreteness of the nonzero spectrum  of Maxwell-Leontovich operators $M_{\imp,a}$ in the case where the impedance coefficient is not a constant function on the absorbing part of the boundary.


The extension to general impedance coefficients $a (\cdot)$ of the Krein-Kulikov proof for the discreteness of $\sigma (M) \setminus \{0\}$ requires not only
 the compact embedding for the space $\Himp^2  \cap \Sem$ into $\Sem$ (see \cite{ACL18,PS22} and \eqref{e:XimpCompImb}),
but also an additional step in the form of
\begin{align} \label{e:ContEmIntro}
& \text{the continuous embedding  $D(\Mo) \imb \Himp^2$ of the domain $D(\Mo)$ of $\Mo$}\\
& \text{into the Hilbert space $\Himp^2 = \Himp \oplus \Himp$,} \notag \\
& \text{where } \quad \Himp = \{ \ubf \in H(\curlm,\Om) \,:\, 
 \n \times \ubf \in \LLt\},  \notag
\end{align}
and where all spaces are equipped with appropriate graph norms (see Sections \ref{ss:MS} and \ref{ss:Disc}).
Without additional assumptions on boundary conditions, \eqref{e:ContEmIntro} generally does not hold,
see Remark \ref{r:DimbHimp}. 

We prove the continuous embedding \eqref{e:ContEmIntro}
and the discreteness $\si (\Mo |_\Sem) = \si_{\disc} (\Mo |_\Sem)$ of the spectrum  for the case of
the generalized impedance boundary conditions 
\begin{gather} \label{e:ZBC}
   \n   \times \Ebf   = Z  [(\n  \times \Hbf )\times \n ] ,
\end{gather}
with \emph{impedance operators} $Z:\LLt \to \LLt$ satisfying the following assumptions: 
\begin{align}
& \text{$Z$ is an accretive bounded operator  in $\LLt$ }  \label{e:Z1} 
\\
  &\text{with the coercivity property \qquad $ \|\vbf\|_\paOm^2 \lesssim |(Z \vbf,\vbf)_\paOm| $, \quad $\vbf\in\LLt$.} \label{e:Z2}
\end{align}
Recall that a linear operator $T$ in a complex Hilbert space $\{X, (\cdot,\cdot)_X\}$  is called \emph{accretive}  if \linebreak $\Re (Tf,f)_X \ge 0$ for all $f \in D(T)$.
By  $(\cdot,\cdot)_{\partial\Omega} $ and $\|\cdot\|_{\partial\Omega}$ the inner product and, respectively (resp.), the norm in the Hilbert space $L^2_t (\partial\Omega)$ are denoted. The notation `$\lesssim$' means that the corresponding inequality is valid after multiplication of the left (or right)  side  on a certain constant $C>0$; i.e.,  \eqref{e:Z2} means $ C \|\vbf \|_\paOm^2 \le |(Z \vbf,\vbf)_\paOm| $ for all $\vbf \in \LLt$. 

Taking the impedance operator $Z$ equal to the operator of multiplication on $a(\cdot)$, one sees
that the class of boundary conditions  \eqref{e:ZBC}-\eqref{e:Z2} includes all Leontovich boundary conditions  \eqref{e:aBC} with $a (\cdot)$ satisfying \eqref{e:aLinf}-\eqref{e:aCo}.
Note that 
the Maxwell operator $M_{\imp,Z}$ defined by \eqref{e:EP}, \eqref{e:ZBC}
is m-dissipative  under assumptions \eqref{e:Z1}-\eqref{e:Z2} (see \cite{EK22} and a simpler proof in Appendix \ref{a:mD}).

The difficulty (A) with the proof of eigenvalue  convergence for nonselfadjoint eigenproblems  is addressed in the present paper for a fixed m-dissipative boundary condition chosen in the class  \eqref{e:Z1}-\eqref{e:Z2}. We obtain the theorem on the convergence of eigenvalues directly (without the use of compact convergence of \cite{OSY92}) combining the properties of the H-convergence with the Helmholtz-Hodge decompositions of $\CC^3$-vector-fields in $\Om$ (see Sections \ref{ss:conv} and \ref{ss:Hodge}). Namely, it is proved in Section \ref{s:iii} that the Murat-Tartar H-convergence \cite{MT78,A02} of material parameters, $\ve_n \Hcon \ve_\infty $ and  $\mu_n \Hcon \mu_\infty$, and the convergence of 
associated eigenvalues $\om_n  \to \om_\infty$ implies that $\om_\infty$ is an eigenvalue of the Maxwell operator $M_{\imp,Z}^{\ve_\infty,\mu_\infty}$ that corresponds to the homogenized material parameters $\ve_\infty$ and $\mu_\infty$ (see Theorem \ref{t:c}). 

As a by-product of this result, we  prove that, for every uniformly positive definite and uniformly bounded family of material parameters $\ve$ and $\mu$, there exists a region $\{ z \in \CC : 0 <|z| < R\}$ that is free of eigenvalues (see Theorem \ref{t:EigFree}).
These results are applied in Sections \ref{ss:OptRes}, \ref{ss:uc}, and \ref{s:discu} to
various eigenvalue optimization problems for Maxwell operators.

One of the nonselfadjoint spectral optimization problems arising in optical engineering and applied physics  is 
the Q-factor maximization for photonic crystal microcavities \cite{LSV03,VS21}, which was partially inspired by the Haroche design of `Schrödinger photonic-cat' experiments  \cite{H13}. 
We consider for this problem a new  optimization formulation, 
which is different, but somewhat related, to the formulations for the minimization of the eigenoscillation (exponential) decay rate $\Dr (\om) = - \Im \om$ that were  proposed before in \cite{HS86,K13,OW13,KKV20} for other types of spectral problems.

Namely, the problem considered in Section \ref{ss:OptRes} is to create a complex eigenvalue $\om$ as close as possible to a specific range $\Ic = [\vphi_-,\vphi_+] \subset \RR$ of real frequencies assuming that
the material parameter pair $\{\ve,\mu\}$ varies in a certain family of composite structures feasible for the fabrication, or in the corresponding relaxed H-closed feasible family.
With the use of Theorems \ref{t:c} and \ref{t:EigFree}, we apply the homogenization method to the question of the existence of optimizers.

The connections between the unique continuation for Maxwell system \cite{EY06,D12,NW12,BCT12} and the optimization of eigenvalues are considered in Sections \ref{ss:uc} and \ref{s:discu}. We would like to note that from the point of view of the  Q-factor maximization, the example of nonunique continuation of \cite{D12} leads to a cavity with the quality-factor $Q=\infty$. 
This naturally raises the question whether the continuous anisotropic coefficients $\ve$ and $\mu$  considered in \cite{D12} belong to the set of H-limits of piecewise constant structures consisting of several materials available for fabrication. This question is connected with the coupled G-closure problem that is discussed in the beginning of Section \ref{s:discu}.

In Appendix, proofs of several important results, which we use as technical tools, are given. These results  may be already known to specialists, but either complete proofs are difficult to find in the publications, or substantially simpler proofs can be given in the specific settings that we use.

\textbf{Notation.} 
Let $I_{\RR^3}$ denote the identity $3\times3$-matrix. We put $\RR_\pm =\{t \in \RR \ : \ \pm t >0 \}$.  By  $\<\cdot,\cdot \> $ we denote the standard inner product in $\CC^3$ or $\RR^3$.
For the functional spaces and basic settings related to Maxwell systems we refer to  \cite{ACL18,EK22,M03} and Sections \ref{ss:MS}-\ref{ss:Disc}; for the  notions of spectral theory related to dissipative and accretive operators to \cite{E12,Kato,RSIV,SFBK10} and Section \ref{ss:SpTh}. 
By $(\cdot,\cdot)_X$ we denote the inner product in a Hilbert space $X$; the notation $(\cdot,\cdot)_\Om$ is the shortening for $(\cdot,\cdot)_{L^2 (\Om,\CC^3)}$, while the notation  
$(\cdot,\cdot)_\paOm$ is the shortening for $(\cdot,\cdot)_\LLt$. By 
 $\<\cdot,\cdot\>_\Om$ we denote the 
sesquilinear  pairing between $H^{-1}(\Omega)$ and $H^1_0(\Omega)$ w.r.t. the pivot space $L^2 (\Omega)$. The sign $\oplus$ stands for an orthogonal sum of Hilbert spaces. 
Let $T:D(T) \subseteq X \to X$ be a (linear) operator in $X$ with a domain $D(T)$. 
The \emph{range} of $T$ is denoted by $\ran T := \{T h : \ h \in D(T)\}$. 
For Banach spaces $X_1$ and $X_2$, we denote by $X_1 \imb X_2$  continuous embeddings,
and by $X_1 \imb \imb X_2$ compact embeddings.


\section{Main results and settings of the paper}\label{s:MR}

\subsection{Spaces and trace maps associated with Maxwell operators}
\label{ss:MS}

It is assumed always that $\Om \subset \RR^3$  is a Lipschitz domain.  This means that $\Om $ is a nonempty bounded open connected set with a Lipschitz continuous boundary $\pa \Om$, see, e.g., \cite{M03,ACL18}.   Let $\RR^{3\times3}$ be the real Banach space of $3\times3$ real-valued matrices (the choice of a norm in $\RR^{3\times3}$ is not important). By $\RR_\sym^{3\times3}$ we denote the linear subspace of $\RR^{3\times3}$ consisting of symmetric matrices.

Let $\al,\beta>0$ be constants such that $0 < \al \le \beta$. Modifying the notation of \cite{A02}, we denote by $\MM_{\al,\beta}$ the set of invertible real matrices $\Ac \in \RR^{3\times3} $ satisfying the following two coercivity conditions 
\begin{equation} \label{e:Mab}
\text{$\al |\ybf|^2  \le \< \Ac \ybf, \ybf\> $ and $\beta^{-1} |\ybf|^2  \le \<  \Ac^{-1} \ybf, \ybf\> $ for all $\ybf \in \RR^3$,}
\end{equation}
where $\<\ybf,\zbf\> = \sum_{j=1}^3 y_j \overline{z_j}$ for $\ybf,\zbf \in \CC^3$.
Let $\MM^\sym_{\al,\beta} $ be the set of symmetric matrices satisfying \eqref{e:Mab},  
\[
\MM^\sym_{\al,\beta}  = \MM_{\al,\beta} \cap \RR_\sym^{3\times3}  = \{ \Ac \in \RR^{3\times3}  \ : \ \al |y|^2  \le \< \Ac y, y\> \le \beta |y|^2 \text{  for all $y \in \RR^3$}\} .
\]
Since $\Ac \in \MM_{\al,\beta}$ implies $|\Ac \ybf |\le \beta^{-1} |\ybf|$ for all $\ybf \in \RR^3$,
the sets $\MM_{\al,\beta}$ and $\MM^\sym_{\al,\beta}$ are bounded in $\RR^{3\times3}$. 
Let $ L^\infty \bigl(\Om,\MM^{(\sym)}_{\al,\beta} \bigr)  = \{ A \in L^{\infty} (\Om,\RR^{3\times3}) \ : \   A(x) \in \MM^{(\sym)}_{\al,\beta}  \text{ for  a.a. } x \in \Om\}  $.
It will be assumed always that the material parameters  $\ve$ (dielectric permittivity) and $\mu$ (magnetic permeability) satisfy 
$  \ve, \mu \in \Mabsym $ with certain $\al$ and $\beta$ such that $0 < \al \le \beta$.

We refer to \cite{M03,ACL18,EK22} for the main functional spaces associated with Maxwell operators. 
The space $ 
H (\curlm,\Om) := \{ u \in \LLOm \ : \ \nabla \times u \in \LLOm \} 
$
and other spaces built as domains of operators are assumed to be equipped with the graph norms. 
Here $H (\curlm,\Om)$ is the domain (of definition) of 
the unbounded operator 
$\curlm : \ubf \mapsto  \nabla \times \ubf$  in the Hilbert space $\LLOm$,
 where $\nabla \times \ubf $ is understood in the distribution sense. Since $\curlm$ is a closed operator, $H (\curlm,\Om)$ is a Hilbert space. The Hilbert space $
H_0 (\curlm, \Om) 
$
is defined as  the closure in $
H (\curlm,\Om) $ of the space $C^\infty_0 (\Om, \CC^3)$ of compactly supported in $\Om$ smooth  $\CC^3$-vector-fields.

The space  $ \Himp = \{ \ubf \in H(\curlm,\Om) \,:\, 
 \n \times \ubf \in \LLt\}
$ is a Hilbert space,  see the proof of \cite[Theorem 4.1]{M03}.
This space provides one of the ways to give a rigorous interpretation to Leontovich boundary conditions \cite{M03,LL04, EK22} and can be equivalently written as 
\[
 \Himp  = \{ \ubf \in H(\curlm,\Om) \,:\, 
 \ubf_\Tan \in \LLt\} , \qquad \text{ where $\ubf_\Tan:=  (\n\times \ubf) \times \n$.}
 \]
The \emph{tangential component  $\ubf_\Tan = (\n\times \ubf) \times \n = -\n \times \n \times (\ubf|_{\pa \Om})$ of the trace $\ubf|_{\pa \Om}$} and the \emph{tangential trace}
$\n \times \ubf = \n  \times (\ubf |_{\pa \Om}) $ for $\ubf \in H (\curl,\Om)$  are understood  in a  generalized sense, see \cite{C96,BCS02,M03,M04,ACL18,EK22}.
 Namely, in the case of a Lipschitz domains $\Om$, one has 
$\n \in L^\infty (\pa \Om, \RR^3)$ for the \emph{unit outward normal vector-field} $\n:\pa \Om \to \RR^3$.
The operator $\n^\times : \ubf \mapsto \n  \times (\ubf |_{\pa \Om})$ and the operator of the tangential projection trace $\pi_\top : \ubf \mapsto \ubf_\Tan $ first defined for $C^\infty (\overline{\Om};\CC^3)$-fields on the closure $\overline{\Om}$ of $\Om$  have unique extensions as continuous (but not surjective)  operators from $H (\curlm, \Om)$ to $H^{-1/2} (\pa \Om,\CC^3)$.
The images of these operators 
\[
\text{$\n^\times H (\curlm, \Om) = H^{-1/2} (\Div_{\pa \Om}, \pa \Om)$
and $\pi_\top H (\curlm, \Om) = H^{-1/2} (\curl_{\pa \Om}, \pa \Om)  $}
\]
can be equivalently  defined
as domains of special extensions of the surface divergence $\Div_{\pa \Om}$ and the surface scalar curl-operator $\curl_{\pa \Om}$ \cite{C96,BCS02,BHPS03,M04} (see also the comments in  \cite{M03,ACL18,EK22}). The role of the spaces $H^{-1/2} (\Div_{\pa \Om}, \pa \Om)$
and $H^{-1/2} (\curl_{\pa \Om}, \pa \Om)  $ is that they are the two trace spaces associated with the integration by parts for the operator $\curlm$.

 The (graph) norm $\| \cdot \|^2_\Himp $ of $\Himp$ is defined via the equalities 
\begin{gather} \label{e:NormHimp}
 \|\ubf\|^2_\Himp = 
 \|\ubf\|^2_{H (\curlm,\Om)}  +\|\n \times \ubf\|^2_\paOm =\|\ubf\|^2_{H (\curlm,\Om)}  +\|\ubf_\Tan\|^2_\paOm ,
\end{gather}
where $\| \cdot \|_\paOm$ denotes the norm in the Hilbert space $L^2_t (\pa \Om)$.
Note that in \eqref{e:NormHimp} we used the following elementary fact: 
the map  $\vbf \mapsto \n \times \vbf$ considered as a linear operator in $\LLt$ is a unitary operator.

\subsection{M-dissipative Maxwell operators with discrete spectra}\label{ss:Disc}

A linear operator $T:D(T) \subseteq X \to X$ in a complex Hilbert space $X$  
 is called \emph{accretive}  if $\Re (Tf,f)_X \ge 0$ for all  $f \in D(T)$ \cite{Kato}, and is called \emph{dissipative} if $\Im (Tf,f)_X \le 0$ for all $f \in D(T)$ (we use here the mathematical physics convention for dissipative operators \cite{E12}).
  A dissipative operator $T$ in a Hilbert space $X$ is called \emph{m-dissipative} if $\CC_+ := \{\om \in \CC \, : \, \Im \om>0\}$ is a subset of its resolvent set $\rho (T)$ and $\|(T-\om)^{-1}\| \le \frac 1{\Im \om}$ for all $\om \in \CC_+$.

In this paper it is assumed that the operator $Z:\LLt \to \LLt$ satisfies conditions \eqref{e:Z1}-\eqref{e:Z2}, i.e., $Z$ is a bounded accretive operator in the Hilbert space $\LLt$ generating a coercive sesquilinear form (for discussions of boundary conditions with more general impedance operators, see \cite{ACL18,EK22}).
We work with generalized impedance boundary conditions \eqref{e:ZBC}, which now, using  the the notation of Section \ref{ss:MS}, can be written shorter  as
$   \n   \times \Ebf   = Z  \Hbf_\Tan $. 

The Maxwell operator $M_{\imp,Z} = M_{\imp,Z}^{\ve,\mu}$ associated with eigenproblem \eqref{e:EP} and the boundary condition $   \n   \times \Ebf   = Z  \Hbf_\Tan $ is m-dissipative \cite{EK22}. Let us give the rigorous definition of $M_{\imp,Z}$.
In the energy space $\Le$ defined by  \eqref{e:Le0}-\eqref{e:Le2}, we consider the symmetric  Maxwell operator $\Mmin$  
defined by the differential expression
 $
 \Mf (\ve,\mu) : \left[ \begin{matrix}\Ebf \\ \Hbf \end{matrix} \right] \mapsto  
 \left[ \begin{matrix} \ \ \ii \ve^{-1} \nabla\times \Hbf 
 \\ - \ii \mu^{-1} \nabla\times \Ebf \end{matrix} \right]
$
on the domain 
\[
\text{$D(M_{0,0}) = H_0 (\curlm,\Omega)^2 $, \quad
where $ H_0 (\curlm,\Omega)^2 =H_0 (\curlm,\Omega) \oplus H_0(\curlm,\Omega)$.}
\]
The adjoint (in $\Le$) operator $\Mmin^*$ 
has the same differential expression $\Mf (\ve,\mu)$, but a different domain
$ D(\Mmin^*) = H (\curlm,\Om)^2$, which is wider and is the maximal natural domain 
for this differential expression in $\Le$.

The operator $M_{\imp,Z} = M_{\imp,Z}^{\ve,\mu}$  is defined as the restriction 
$ M_{\imp,Z} := \Mmin^* \!\uph_{D(M_{\imp,Z})}  $ of the operator $\Mmin^*$
 to $D(M_{\imp,Z}) := \{ \{\Ebf,\Hbf\} \in \Himp^2  :  \n \times \Ebf   = Z  \Hbf_\Tan  \}$.
The notation $M_{\imp,Z}^{\ve,\mu}$ (and the notation $M_{0,0}^{\ve,\mu}$ for the operator $M_{0,0}$) will be used in the cases, where it is important to emphasize the dependence on the material parameters  $\ve $ and $\mu $.

\begin{ex}[Maxwell-Leontovich operators]\label{e:LMop}
The Maxwell-Leontovich operator $ M_{\imp,a} $ defined in Section \ref{s:I} can be written as   $M_{\imp,Z} $ with $Z=\mf_a$, where $\mf_a : \vbf \mapsto a \vbf$ is the operator of multiplication on $a(\cdot)$ in $\LLt$. In this case, the assumptions \eqref{e:aLinf}-\eqref{e:aCo} on the impedance coefficient $a(\cdot)$ are equivalent to  the 
assumptions \eqref{e:Z1}-\eqref{e:Z2} (this follows, e.g.,  from  \cite[Remark 4.2.7]{ACL18}).
\end{ex}

\begin{remark} \label{r:MZmDis}
Since  $M_{\imp,Z} $ is m-dissipative in $\Le$, the Maxwell-Leontovich operator $ M_{\imp,a} $ is m-dissipative if assumptions \eqref{e:aLinf}-\eqref{e:aCo} are satisfied. This result was obtained in \cite{EK22} as a by-product of a description of all time-independent linear m-dissipative boundary conditions in terms of  m-boundary tuples. We give in Appendix \ref{a:mD} an elementary proof of the m-dissipativity of $M_{\imp,Z} $, which, in the style of \cite{LL04}, is based on the Lax–Milgram lemma.
\end{remark}

The domain $D(M_{\imp,Z})$ of $M_{\imp,Z}$ is equipped with the graph norm for $M_{\imp,Z}$ in  $\Le$, i.e.,
\[
\|\{\Ebf,\Hbf\}\|_{D(M_{\imp,Z})}^2 = (\ve \Ebf,\Ebf)_\Om+(\mu \Hbf,\Hbf)_{\Om}  + \|\nabla \times \Ebf\|_\Om^2+\|\nabla \times \Hbf\|_\Om^2.
\]
Note that this norm is generated by the graph norm for $\Mmin^*$ and is independent of  $Z$.
Recall that 
\[
 \|\{\E,\H\}\|_{H (\curl,\Om)^2}^2 =  \|\E\|_{H (\curl,\Om)}^2 + \|\H\|_{H (\curl,\Om) }^2 =
\| \E\|_\Om^2+\| \H\|_\Om^2  + \|\nabla \times \E\|_\Om^2+\|\nabla \times \H\|_\Om^2.
\]
Since $\ve,\mu \in \Mabsym$,
\begin{equation} \label{e:NDeqHcurl2}
\text{the norms $\|\cdot\|_{D(M_{\imp,Z})}$ and $\| \cdot \|_{H (\curl,\Om)^2}$ are equivalent in the space $D(M_{\imp,Z})$.}
\end{equation}
The definition \eqref{e:NormHimp} of $\| \cdot \|_{\Himp}$ implies
\begin{gather} \label{e:NHimp2}
\|\{\E,\H\}\|_{\Himp^2}^2 = \|\E\|_{H (\curl,\Om)}^2 + \|\H\|_{H (\curl,\Om) }^2 + \| \E_\Tan \|_\paOm^2 + \| \H_\Tan \|_\paOm^2  \ .
\end{gather}

\begin{theorem}\label{t:b}
Suppose \eqref{e:Z1}-\eqref{e:Z2}. Then:
\item[(i)]  
$\quad  \|\E_\Tan\|^2_{\partial\Omega} +
  \|\H_{\Tan}\|^2_{\partial \Omega} \lesssim 
  \|\E\|^2_{H(\curl,\Omega)} + 
  \|\H\|^2_{H(\curl,\Omega)} \ $ \quad for all $\{\E,\H\} \in D(M_{\imp,Z})$.
\item[(ii)] The following continuous embedding holds 
\begin{equation} \label{e:DimbHimp}
D( M_{\imp,Z}) \ \imb \ \Himp^2  \ .
\end{equation} 
\end{theorem} 

The proof of  this theorem is given in Section \ref{s:ProofsDisc}
(note that statement (ii) follows immediately from statement  (i) combined with  \eqref{e:NDeqHcurl2} and \eqref{e:NHimp2}). 
The role of Theorem \ref{t:b} is that \eqref{e:DimbHimp} is one of the two components of the proof of the discreteness of the nonzero spectrum $\si (M_{\imp,Z}) \setminus \{0\}$ of the operator $M_{\imp,Z}$, see Theorem \ref{t:DiscSp} and Section \ref{s:ProofsDisc}.

Example \ref{e:LMop} implies that, for the case of Maxwell-Leontovich operators $M_{\imp,a}$, the estimate (i) of Theorem \ref{t:b}  is valid for $\{\Ebf,\Hbf\} \in D(M_{\imp,a})$ under assumptions \eqref{e:aLinf}-\eqref{e:aCo}.

\begin{remark}\label{r:DimbHimp}
 While the impedance operator $Z$ does not influence the norm in 
$D (M_{\imp,Z})$, it influences the choice of the closed subspace $D (M_{\imp,Z})$ of 
$H (\curl,\Om)^2$. It is easy to see that statements (i) and (ii) of Theorem \ref{t:b} are not valid 
if $D (M_{\imp,Z})$ is replaced by the whole $H (\curl,\Om)^2$, and so, Theorem \ref{t:b} is  generally not valid without assumptions  \eqref{e:Z1}-\eqref{e:Z2}. A simple counterexample is provided by the 'extremal  Maxwell-Leontovich operator' $M_{\imp,0}$ with the identically zero impedance coefficient $a \equiv 0$. In this case the multiplication operator $\mf_0 $ of Example \ref{e:LMop} is the zero operator in $\LLt$,
and $D(M_{\imp,0}) = H_0 (\curlm,\Om) \times \Himp $. The estimate (i) of Theorem \ref{t:b} and,  particularly, the estimate $ \|\Hbf_{\Tan}\|^2_{\partial \Omega} \lesssim \|\Hbf\|^2_{H(\curl,\Omega)} $ 
  are not valid on $D(M_{\imp,0})$
since the Hilbert space $\pi_\top H (\curlm,\Om) = H^{-1/2} (\curl_{\pa \Om}, \pa \Om)  $ is not embedded into $\LLt$ (see \cite{C96,BCS02,EK22}). 
\end{remark}

For a matrix-function $\xi  \in \Mabsym $ (later $\xi  = \ve $ or $\xi = \mu$), the 'weighted' complex  Hilbert space
$L^2_\xi (\Om,\CC^3)$  by definition  coincides with $L^2 (\Om,\CC^3)$ as a linear space, but is equipped with the `weighted' inner product 
$ (\ubf,\wbf)_{L^2_\xi (\Om,\CC^3)} := (\xi \ubf,\wbf)_\Om $. 
The norms in $L^2_\xi (\Om,\CC^3)$ and $L^2 (\Om,\CC^3)$ are equivalent (due to \eqref{e:Mab}).
The space of $\CC$-valued functions $H^1_0 (\Om) := \{ f \in H^1 (\Om) \, : \, f |_\paOm = 0 \}$ is a closed subspace of the standard complex Hilbertian Sobolev space $H^1 (\Om)$.

We use the (weighted)  \emph{orthogonal  Helmholtz--Weyl decomposition} \cite{GR86,C96,M03,ACL18}
\begin{gather}  \label{e:HWxi}
L_\xi^2 (\Omega,\CC^3) \ = \  H(\Div \xi 0,\Omega) \ \oplus \ \grad H^1_0 (\Omega)
, \\
\text{ where  } 
 H(\Div \xi 0,\Om) =\{ \ubf \in L_\xi^2(\Om,\CC^3)  \;:\; \text{$\nabla\cdot (\xi \ubf)=0$ in $\Om$ in the sense of distributions}\}
\label{e:HdivXi0}
\end{gather}
and 
$\grad H^1_0 (\Om) = \{\grad \psi : \psi \in H^1_0 (\Om)\}$ are  closed subspaces of $ L_\xi^2 (\Omega,\CC^3)$
(see more details in Section \ref{ss:SpTh} and Appendix \ref{a:HdecOm}).
Since $\Le = L_\ve^2 (\Omega,\CC^3)  \oplus L_\mu^2 (\Omega,\CC^3) $,
the Helmholtz-Weyl decompositions of  $L_{\ve (\mu)}^2 (\Omega,\CC^3) $  lead to the orthogonal decomposition of the energy space:
\begin{gather*}
 \Le = \Sem \oplus \GG_{0} (\Om), \quad \text{ where $\GG_{0} (\Om) = \grad H^1_0 (\Omega) \oplus \grad H^1_0 (\Omega) $.}
\end{gather*} 
The 1st space $\grad H^1_0 (\Omega)$ in the orthogonal decomposition of $ \GG_{0} (\Om)$ has to be understood as a closed subspace of 
$L_\ve^2 (\Omega,\CC^3)$ (equipped with the norm $\| \cdot \|_{L_\ve^2 (\Omega,\CC^3)}$) and the 2nd space $\grad H^1_0 (\Omega)$ analogously as a closed subspace of $L_\mu^2 (\Omega,\CC^3)$. The space $\Sem$ is defined by \eqref{e:Sem}.

The orthogonal decomposition 
$\Le = \Sem \oplus \GG_{0} (\Om)$ reduces  the m-dissipative operator $M_{\imp,Z}$ to the orthogonal sum 
$ \quad
M_{\imp,Z} = M_{\imp,Z} |_\Sem \oplus 0 \quad \text{ (see \cite{EK22})},
\quad 
$
where 
the part $M_{\imp,Z} |_{\GG_{0} (\Om)} $ of $M_{\imp,Z}$ in the reducing subspace $\GG_{0} (\Om)$
is the zero operator in $\GG_{0} (\Om)$ (the definition of reducing subspaces is given in Section \ref{ss:SpTh}). 
Moreover, since the operator $M_{\imp,Z}$ is m-dissipative in $\Le$, its part $M_{\imp,Z} |_\Sem$ in the reducing subspace $\Sem$ is an m-dissipative operator in $\Sem$ (see  \cite[Remark 2.2]{EK22} for the details).


Note that, for any $\ve,\mu \in \Mabsym$ and any impedance operator $Z$, the infinite-dimensional space $\GG_{0} (\Om)$ is a subspace of the kernel $\ker M_{\imp,Z} = \{ \Phi \in D(M_{\imp,Z})  : \ M_{\imp,Z} \Phi = 0 \}$. So,  the eigenvalue $\om =0$ of the operator $M_{\imp,Z}$ has infinite geometric and algebraic multiplicities (in the sense of \cite{Kato})
and $\om = 0$ is a point of the essential spectrum $\si_\ess (M_{\imp,Z}) = \si (M_{\imp,Z}) \setminus \si_\disc (M_{\imp,Z})$ of $M_{\imp,Z}$
(in the sense of \cite{RSIV}).

The basic definitions of spectral theory including definitions of operators with compact resolvent and the discrete spectrum $\si_\disc (T)$ are collected below in Section \ref{ss:SpTh}.
The following theorem on discrete spectra of Maxwell operators $M_{\imp,Z}$  is the main result of this subsection. 

\begin{theorem}[discreteness of spectra] \label{t:DiscSp}
Let an impedance operator $Z$ satisfy \eqref{e:Z1}-\eqref{e:Z2}. Then:
\item[(i)] The restricted Maxwell operator $M_{\imp,Z} |_{\Sem} $ is an m-dissipative operator with compact resolvent 
and purely discrete spectrum.
\item[(ii)] 
The spectrum of $M_{\imp,Z}$ consists of isolated eigenvalues. An eigenvalue $\om$ of $M_{\imp,Z}$ has a finite algebraic multiplicity if and only if $\om \neq 0$. 
\end{theorem}    

The proof of this theorem is given in Section \ref{s:ProofsDisc}.


 \subsection{Convergence of eigenvalues under H-convergence of material parameters}\label{ss:conv}

If this is not explicitly stated differently, the impedance operator $Z:\LLt \to \LLt$ is assumed to be fixed in the rest of the paper and satisfies conditions \eqref{e:Z1}-\eqref{e:Z2}.

The notation '$f_n \Weak f$ in $X$' 
 means that a sequence $ \{f_n\}_{n\in \NN} \subset X$ converges weakly in a Banach space $X$.
By  $f_n \to f$, we denote the strong convergence. 
The notation 
$A_n \Hcon A$ is used for the H-convergence  of Murat \& Tartar introduced in  \cite{MT78}  for 2nd order elliptic operators. The following slightly different, but equivalent, definition of H-convergence is from the monograph of Allaire \cite{A02}.

 
\begin{definition}[\cite{A02}, cf. \cite{MT78,MT85,OSY92,T09}] \label{d:A02}
 A sequence $\{A_n\}_{n\in \NN} \subseteq \Mab$ is said to H-converge to an H-limit (homogenization limit) $A_\infty \in \Mab $ as $n \to \infty$ if for any $f \in H^{-1}(\Omega,\RR)$, the sequence $\{u_n\}_{n\in \NN}$ of $H^1_0 (\Om,\RR)$-solutions  to the Dirichlet boundary value problem
\begin{equation*}
 \nabla \cdot (A_n\nabla u)  \ = \ f \ \mbox{ in } \Omega \;, \qquad
 u|_\paOm  \ = \ 0 ,
\end{equation*}
satisfies $u_n \Weak u_\infty$ in $H^1_0(\Omega,\RR)$ and $A_n\nabla u_n \Weak A \nabla u_\infty$ in $L^2(\Omega,\RR^3)$, where $u_\infty \in H^1_0(\Omega,\RR)$ is the solution to the Dirichlet boundary value problem
$ \nabla\cdot (A_\infty \nabla u) \ = \ f $, $ u|_\paOm  \ = \ 0$.
\end{definition}

Note that, for a sequence  $\{A_n\}_{n\in \NN} \subseteq \Mabsym$ of symmetric matrix-functions, the H-convergence is equivalent to the G-convergence introduced  by De Giorgi \& Spagnolo (see \cite{MT78,D93,A02}).

\begin{remark} \label{r:Hcon}
There exist \cite{A02} infinitely many metrics (distance functions) $\rhoH $ on $\Mabsym$ such that  the H-convergence in $\Mabsym$  becomes the convergence of sequences w.r.t. $\rhoH$. In what follows, we assume that one such metric 
$\rhoH $ is fixed. By the result of Murat \& Tartar \cite{MT78} (see also  \cite{A02}),  
the metric space $(\Mabsym, \ \rhoH)$ is compact.
\end{remark}

While Definition \ref{d:A02} is written in a nonlocal form, the  H-convergence is actually a local property (see  \cite{MT78,A02}). We will use the following equivalent characterization \cite{T85} of the H-convergence, which is independent of boundary conditions.

 \begin{prop}[\cite{T85}, see also \cite{M23}] \label{p:Hcon}
A sequence $\{A_n\}_{n \in \NN} \subseteq \Mab$ is H-convergent to  $A_\infty \in \Mab$ if and only if, for every two sequences $\{\Ebf^n \}_{n \in \NN}$ and $\{\Dbf^n \}_{n \in \NN}$
 in   $L^2(\Omega,\RR^3)$, the combination of the following properties 
  \begin{align}
  &\text{$\Dbf^n=A_n \Ebf^n$ a.e. in $\Om$ for all $n \in \NN$},  \label{e:DE1}\\
  &\text{$\Ebf^n \Weak \Ebf^\infty$ in $L^2(\Omega,\RR^3)$, \qquad 
  $\Dbf^n \Weak \Dbf^\infty$ in $L^2(\Omega,\RR^3)$,} \label{e:DE2}\\
  &\text{ $\{\nabla \times \Ebf^n\}_{n\in \NN}$ is a relatively compact subset of $H^{-1}(\Omega,\RR^3)$,} \label{e:DE3}\\
  &\text{$\{\nabla\cdot \Dbf^n\}_{n \in \NN}$ is a relatively compact subset of $H^{-1}(\Omega)$,} \label{e:DE4}
  \end{align} 
implies $\Dbf^\infty =A_\infty \Ebf^\infty$ a.e. in $\Om$.
\end{prop} 

It is difficult to find a proof of this proposition published in English. That is why, we include such a proof   in Appendix \ref{s:EquivHconv}.
We apply Proposition \ref{p:Hcon} to complex $L^2(\Omega,\CC^3)$-eigenfields of m-dissipative Maxwell operators. Therefore,  the following remark is needed.

\begin{remark} \label{r:HconC}
Since the values of $A_n (\cdot)$ are real matrices, Proposition \ref{p:Hcon} remains valid if the sequences $\{\Ebf^n \}_{n=1}^\infty$ and $\{\Dbf^n \}_{n=1}^\infty$ are allowed to be in the complex space  $L^2(\Omega,\CC^3)$.
\end{remark}

Under a nontrivial solution $\{\Ebf , \Hbf \}$ to the eigenproblem 
 \begin{equation*}\label{e:EPbc}
     \ii  \nabla \times \Hbf  = \omega \ve \Ebf, \qquad -\ii  \nabla \times \Ebf = \omega \mu \Hbf , \qquad \n \times \Ebf   = Z  \Hbf_\Tan,
     \end{equation*}
we understand an eigenfield $\{\Ebf , \Hbf \}$ of the m-dissipative operator $M_{\imp,Z}$
(`nontrivial' in this context means that the eigenfield  $\{\Ebf , \Hbf \}$ is always supposed to have  positive  energy  $\| \{\Ebf ,\Hbf \} \|_{\Le}^2 >0$). 
A vector-field $\{\Ebf , \Hbf \}$ is said to be \emph{$\Le$-normalized} if $\| \{\Ebf ,\Hbf \} \|_{\Le} = 1$.

Under a subsequence  of $\NN$, we understand a strictly increasing to $+\infty$ sequence of indices $\{n_k\}_{k \in \NN} \subseteq \NN$. Whenever we iteratively pass to  (sub-)subsequences $\{n_{k_j}\}$, it is assumed that $\{k_j\}_{j \in \NN} \subseteq \NN$ is strictly increasing. With a conventional abuse of the notation, we sometimes keep the same indexing $\{n\}$ or $\{ n_k \}$  for (sub-)subsequences.

\begin{theorem}[convergence of eigenvalues]
\label{t:c}
Assume that $\ve_n , \mu_n \in \Mabsym$, $n \in \NN$, be such that $\ve_n \Hcon \ve_\infty $ and  $\mu_n \Hcon \mu_\infty$ as $n \to \infty$. Assume that, for each $n \in \NN$, $\om_n \neq 0$ and the pair $\{\Ebf^n ,\Hbf^n \}$ is an $\LL_{\ve_n,\mu_n}^2 (\Om)$-normalized solution to the eigenproblem 
 \begin{equation*} 
     \ii  \nabla \times \Hbf^n  = \omega_n \ve_n \Ebf^n, \qquad -\ii  \nabla \times \Ebf^n = \om_n \mu_n \Hbf^n , \qquad \n \times \Ebf^n   = Z  \Hbf^n_\Tan.
     \end{equation*}
Assume that the corresponding eigenvalues $\om_n $ converge to a certain $\om_\infty \in \CC$ as $n \to \infty$.

Then there exists a subsequence $\{n_k\}_{j=1}^\infty \subseteq \NN$ and 
an $\LL_{\ep_\infty,\mu_\infty}^2 (\Om)$-normalized solution $\{\Ebf^\infty ,\Hbf^\infty \}$ to the eigenproblem 
 $    \ii  \nabla \times \Hbf^\infty  = \omega_\infty \ve_\infty \Ebf^\infty, \quad -\ii  \nabla \times \Ebf^\infty = \omega_\infty \mu_\infty \Hbf^\infty , \quad \n \times \Ebf^\infty   = Z  \Hbf^\infty_\Tan $ 
 such that   $\Ebf^{n_k} \Weak \Ebf^\infty$ and $\Hbf^{n_k} \Weak \Hbf^\infty$  in $L^2 (\Om,\CC^3)$ as $k \to \infty$.
\end{theorem}  

Theorem \ref{t:c} is proved in Section \ref{s:iii} using the compactness of $(\Mabsym, \ \rhoH)$
and the (Helmholtz--)Hodge decompositions in $\pa \Om$ and in $\Om$.

\begin{remark}\label{r:Mn}
Theorem \ref{t:c} states, in particular, that the convergence of $\om_n \in \si (M_{\imp,Z}^{\ve_n,\mu_n}) \setminus \{0\}$ to 
$\om_\infty  $ implies that $\om_\infty $ is an eigenvalue of a homogenized Maxwell operator $M_{\imp,Z}^{\ve_\infty,\mu_\infty}$.
Note that, in the case $\lim_{n \to \infty} \om_n = 0$, the conclusion that $\om_\infty =0$ is an eigenvalue of $M_{\imp,Z}^{\ve_\infty,\mu_\infty}$ is trivial, since $0$ is an eigenvalue of every Maxwell operator $M_{\imp,Z}^{\ve,\mu}$ defined in Section \ref{ss:Disc} (due to the inclusion $\GG_{0} (\Om) \subseteq \ker M_{\imp,Z}$). However, from the weak convergence of eigenfields in Theorem \ref{t:c},  it is possible to obtain the following result.
\end{remark}

\begin{theorem}[eigenvalue-free region]
\label{t:EigFree}
There exists a constant $R = R (\al, \beta, \Om, Z)>0$ (depending on $\al$, $\beta$, $\Om$, and $Z$) such that 
\[
\sigma (M_{\imp,Z}) \cap \{\om \in \CC \ : \ 0<|\om |<R\} \ = \ \varnothing \quad \text{for all $\ve,\mu \in \Mabsym$}.
\]
(That is, the punctured disc $\{\om \in \CC \ : \ 0<|\om |<R\}$ is an eigenvalue-free region 
for all Maxwell operators $M_{\imp,Z}^{\ve,\mu}$ with $\ve,\mu \in \Mabsym$.)
\end{theorem}

The proof of Theorem \ref{t:EigFree} is given in Section \ref{s:iii}. It is based on Theorem \ref{t:c} and the decomposition of $M_{\imp,Z}$ in selfadjoint and completely nonselfadjoint parts \cite{SFBK10} (see Proposition \ref{p:mDRed}).

Let $(\Mabsym, \ \rhoH)$ be the metric space with the H-convergence considered in Remark \ref{r:Hcon}.
On the set $\Mabsym^2 = \Mabsym \times \Mabsym$ of pairs of material parameters 
we define the metric 
$ 
\wrhoH (\{\ve_1,\mu_1\},\{\ve_2,\mu_2\}) = \rhoH (\ve_1,\ve_2) +  \rhoH (\mu_1,\mu_2) .
$
The compactness of $(\Mabsym, \ \rhoH)$ implies that 
\begin{equation} \label{e:M2comp}
\text{the metric space $\left( \Mabsym^2, \ \wrhoH \right)$ is also compact}.
\end{equation}
When this does not leads to ambiguity, the convergence of a sequence of material parameter pairs $\{\ve_n,\mu_n\}$, $n \in \NN$, w.r.t. the metric $ \wrhoH $ will be also called H-convergence. 

The following corollary is needed for optimization of eigenvalues of Maxwell operators $M_{\imp,Z}^{\ve,\mu}$.

\begin{corollary}\label{c:conv}
Suppose that a set $\FF \subseteq \Mabsym^2$ of material parameter pairs is H-closed (i.e., closed w.r.t. the metric $\wrhoH$). Then:
\item[(i)]  the set $ \Si [\FF] :=  \bigcup\limits_{\{\ve,\mu\}\in \FF} \si (M_{\imp,Z}^{\ve,\mu}) $ of achievable eigenvalues over $\FF$ is closed in $\CC$.
\item[(ii)] the set $ \Si [\FF] \setminus \{0\} $ is closed in $\CC$.
\end{corollary}  

\begin{proof}
The corollary follows immediately from Theorems \ref{t:c} and \ref{t:EigFree}.
\end{proof}

\subsection{Applications to optimization of eigenvalues}
\label{ss:OptRes}

The family  $\FF \subseteq \Mabsym^2$ of the material parameter pairs $\{\ve,\mu\}$, over which the optimization is performed, will be called the \emph{feasible family}. 
A Lipschitz domain $\Om$, an  impedance operator $Z$ satisfying \eqref{e:Z1}-\eqref{e:Z2}, and the associated m-dissipative boundary condition $   \n   \times \Ebf   = Z  \Hbf_\Tan $ are assumed to be fixed in the process of optimization. 

Let $\{\Ebf,\Hbf\}$ be a certain eigenfield corresponding to an eigenvalue $\om$ of an operator $M_{\imp,Z}^{\ve,\mu}$.
Then $\Dr (\om) = - \Im \om$ is the decay rate of the corresponding eigenoscillations $\{e^{-\ii \om }\Ebf, e^{-\ii \om }\Hbf\}$,
the physical meaning of the real part $\Re \om$ is the (real-)frequency of these eigenoscillations.

If there exists 
$\{ \ve , \mu \} \in \FF$ such that $\om \in \si (M_{\imp,Z}^{\ve,\mu})$, we say that 
$\om$ is an \emph{achievable eigenvalue} (over $\FF$).
The  set $\Si [\FF] =  \bigcup\limits_{\{\ve,\mu\}\in \FF} \si (M_{\imp,Z}^{\ve,\mu}) $   is the set of all achievable eigenvalues.  

Let $\Ic = [\vphi_-,\vphi_+]$, where  
$\vphi_\pm \in \RR$ satisfy $ \vphi_-\le \vphi_+$. 
We define 
\[
\text{ the optimization objective $d_\Ic (\cdot)$ as the distance } \quad d_\Ic (\om) := \min_{\vphi_- \le \vphi \le \vphi_+ } |\om - \vphi| \quad  \text{}
\]
from $\om \in \CC$ to the interval $\Ic$. We say that an achievable eigenvalue $\om_* \in \si (M_{\imp,Z}^{\ve_*,\mu_*})$ and a corresponding material parameter pair $\{ \ve_* , \mu_* \} \in \FF$ are \emph{$d_\Ic$-minimal} over $\FF$ if \ 
$
d_\Ic (\om_*)  = \inf\limits_{\om \in \Si [\FF] }  d_\Ic (\om) .
$
A $d_\Ic$-minimal pair $\{ \ve_* , \mu_* \} $ will be called also an \emph{optimizer} (over $\FF$).

\begin{corollary}[existence of optimizers] \label{c:Opt}
Let $\FF$ be an H-closed subset of $\Mabsym^2$, and  let $\Ic \subset \RR$ be a closed bounded interval. Then there exists at least one $d_\Ic$-minimal  pair $\{ \ve_* , \mu_* \} \in \FF$. 
\end{corollary}  
\begin{proof}
The corollary follows immediately from statement (i) of Corollary \ref{c:conv}.
\end{proof}

\begin{remark}\label{r:Gclosure}
 H-closed feasible families are typically  obtained in the homogenization theory as H-closures of composite structures consisting of several materials available for fabrication
\cite{T85,MT85,KS86,LC86,D93,R01,A02,T09,M23} (see also the bibliography in the monographs \cite{D93,C00,A02,T09,M23}). The description of H-closures in many practical situations was reduced by  Dal Maso \& Kohn (unpublished work, see  \cite{R01,A02})  to the $G$- and $G_\theta$-closure problems of periodic homogenization.
 For the case of two-phase composites, an explicit solution for the conductivity $G_\theta$-closure problem
was obtained independently by Murat \& Tartar \cite{MT85,T85} and Lurie \& Cherkaev \cite{LC86} (see also the references in \cite{D93,A02}). With the use of this explicit solution, it is possible to describe explicitly 
the H-closure for the typical idealized models of a vacuum-silicon photonic crystal, as it is done in the following example.
\end{remark}

\begin{ex} \label{ex:PhC1} (i)
Composite structures consisting of two materials with  isotropic  (relative magnetic)
permeabilities $\wh \mu_1, \wh \mu_2 \in \RR_+$ and isotropic (relative dielectric) permittivities $\ep_1, \ep_2 \in \RR_+$  are modelled by the two disjoint measurable subsets $\Om_1$ and $\Om_2 = \Om \setminus \Om_1$ of $\Om$.
In the natural (Heaviside–Lorentz) system of units the resulting permeability and permittivity matrix-functions are $\mu (x) = (\wh \mu_1 \chi_{\Om_1} (x)  + \wh \mu_2 \chi_{\Om_2} (x)) I_{\RR^3}$ and $\vep (x) = (\ep_1 \chi_{\Om_1} (x)  + \ep_2 \chi_{\Om_2} (x)) I_{\RR^3}$, resp.,
where $\chi_S (\cdot)$ is the indicator-function of a set $S$ (i.e., $\chi_S (x) =1$ if $x \in S$, and $\chi_S (x) = 0 $ if $x \not \in S$). A typical model of a photonic crystal takes the values
$\wh \mu_2 = 1$, $\ep_2 \approx 11.68$ for the 'silicon region' $\Om_2$ (the conductivity and the magnetic susceptibility of silicon are neglected),  and the values $\wh \mu_1 = \ep_1 =1$ for the remaining 'vacuum region' $\Om_1 = \Om\setminus \Om_2$. 
Denoting the family of the corresponding structures by $\FF^0$, one gets
$\FF^0 = \left\{ \{\vep,\mu\} \ : \ \mu \equiv I_{\RR^3}, \ \vep \in \FF_{1,\ep_2}\right\}$,
where 
\[
 \FF_{\ep_1,\ep_2} := \{ \vep (\cdot) =  \ep_1 I_{\RR^3} + (\ep_2 - \ep_1) \chi_{\Om_2} (\cdot) I_{\RR^3} \ : \ \Om_2 \subseteq \Om \ \text{is measurable} \}.
\]
Taking the H-closure of $\FF^0$ (w.r.t. the metric $\wh \rho _\Hr$), we obtain  the corresponding relaxed feasible family $\FF =  \bigl\{ \{\vep,I_{\RR^3}\} \ : \  \vep \in \overline{\FF}^\Hr_{\ep_1,\ep_2}\bigr\} $, which can be used in Corollary \ref{c:Opt}. 
Here $\overline{\FF}^\Hr_{\ep_1,\ep_2}$ is the H-closure of $\FF_{\ep_1,\ep_2}$ (w.r.t. the metric $\rho_\Hr$). That is, $\overline{\FF}^\Hr_{\ep_1,\ep_2}$ is the H-closure in the usual 'conductivity' sense of \cite{MT78,T85,MT85,A02}.
A  descriptions of the family $\overline{\FF}^\Hr_{\ep_1,\ep_2}$ can be obtained \cite[Section 11.2.1]{C00}  by taking the union of the  $G_\theta$-closures (with prescribed ratios of materials) explicitly derived  in \cite{MT85,T85,LC86}. Namely, \quad
$
\overline{\FF}^\Hr_{\ep_1,\ep_2} = \{ \vep \in L^\infty (\Om,\MM^\sym_{\ep_1,\ep_2}) \ : \vep (x) \in \bigcup_{\theta \in [0,1]} \MM_\theta \ \text{ for a.a.  } x \in \Om \},
$
where the sets $\MM_\theta  \subseteq \MM^\sym_{\ep_1,\ep_2} $ introduced in \cite{MT85,LC86} are defined in the following way.
Let us denote by $\Tr \Ac = \sum_{j=1}^3 \Ac_{jj}$ the trace of a matrix  $\Ac \in \RR^{3 \times 3}$,  by '$\le $' the standard partial order in the space $\RR^{3 \times 3}_\sym$ of symmetric matrices, by $\la_\theta^- = (\theta \ep_1^{-1}+ (1-\theta)\ep_2^{-1})^{-1}$ and $\la_\theta^+ = \theta \ep_1 + (1-\theta) \ep_2$ the weighted  harmonic and arithmetic means, resp., of $\ep_1$ and $\ep_2$. Then, for each $\theta \in [0,1]$, the set  $\MM_\theta$  consists of  all symmetric matrices $\Ac \in \RR^{3 \times 3}_\sym$ such that 
$\la_\theta^- I_{\RR^3} \le \Ac \le \la_\theta^+ I_{\RR^3},$ \quad $\Tr ( (\Ac - \ep_1 I_{\RR^3})^{-1}) \le \frac1{\la_\theta^- - \ep_1} + \frac2{\la_\theta^+ - \ep_1},$ \quad \\
 and \quad  $\Tr ((\ep_2 I_{\RR^3} - \Ac)^{-1}) \le \frac1{\ep_2 - \la_\theta^-} + \frac2{\ep_2 - \la_\theta^+} $. 

(ii) We do not know if real $\om \neq 0$ can be an achievable eigenvalue over 
$\FF =  \bigl\{ \{\vep,I_{\RR^3}\} \ : \  \vep \in \overline{\FF}^\Hr_{\ep_1,\ep_2}\bigr\} $ for impedance operators $Z$ with invertible real parts $\Re Z$ (see the assumption  \eqref{e:Zsec} below).
Section \ref{ss:uc} reduces this question to the unique continuation property by means of Proposition \ref{p:real}. However, it follows from Theorem \ref{t:EigFree} that there exists a maximal open interval $\Ic_{\max} = (0,\vphi_+^{\max}) \subseteq \RR_+$ with $ \vphi^{\max}_+ \in (0,+\infty]$ that does not contain achievable over $\FF$ eigenvalues. For any $\Ic \subset \Ic_{\max}$, the $d_\Ic$-minimization  problem considered above is nondegenerate in the sense that $ \min\limits_{\om \in \Si [\FF] }  d_\Ic (\om) >0$. 
\end{ex}

\begin{remark}
For many applied problems, the zero eigenvalue of operators $M_{\imp,Z}^{\ve,\mu}$ has no special significance. The part (ii) of Corollary \ref{c:conv} allows one to exclude the zero eigenvalue from the $d_\Ic$-optimization problem, see Section \ref{s:discu}. Some extreme cases of such optimization problems are demonstrated by Corollary \ref{c:OptSA} below.
\end{remark}

Assume that,  in the Hilbert space $\LLt$,
\begin{equation} \label{e:iZsa}
\text{
one of the operators  $\ii Z$ or $(-\ii) Z$ is  positive definite, selfadjoint,  and  bounded.
}
\end{equation}
Then, Maxwell operators $M_{\imp,Z}^{\ve,\mu}$ are selfadjoint \cite{EK22} (this follows also from Proposition \ref{p:weak} below). Note that \eqref{e:iZsa} implies  \eqref{e:Z1}-\eqref{e:Z2}. By Theorem \ref{t:DiscSp}, $\sigma(M_{\imp,Z}^{\ve,\mu})$ consists of  eigenvalues.

\begin{ex}[superconductor approximation]
The extremal case of Leontovich boundary condition $   \n   \times \Ebf   = - \ii c \Hbf_\Tan $ with a constant $c\in \RR_+$ fits \eqref{e:iZsa} and corresponds to a model with a superconductor encircling $\Om$, see  Landau, Lifshitz \& Pitaevskii \cite[Section 87]{LL84}.  Note that  cooled superconducting niobium optical cavities  were used by Goy, Raimond, Gross \&  Haroche in experiments demonstrating the Purcell effect of enhanced spontaneous emission \cite{H13}.
\end{ex}

\begin{corollary}[optimization of positive and negative eigenvalues] \label{c:OptSA}
Let an impedance operator $Z$ satisfying \eqref{e:iZsa} be fixed.
Let $\FF$ be an H-closed subset of $\Mabsym^2$. Then:
\item[(i)] A minimal  positive eigenvalue is achieved over $\FF$ (in the sense that there exists $\om_*^+ >0$ and  $\{\ve_*,\mu_*\} \in \FF$ such that $\om_*^+ \in \si (M_{\imp,Z}^{\ve_*,\mu_*}) $ and $ \om_*^+ = \min \{\om > 0 \ : \ \om \in \Si [\FF] \}$).
\item[(ii)] Analogously, a maximal negative eigenvalue $\om_*^-$ is achieved over $\FF$.
\end{corollary}
\begin{proof}
We fix  $\{\ve,\mu\} $ for a time being and show that $\si (M_{\imp,Z})$ contains both positive and negative eigenvalues.
The closed symmetric  Maxwell operator $\Mmin$ associated with  the boundary conditions  $0 =   \n   \times \Ebf   =  \Hbf_\Tan $ was defined in Section \ref{ss:Disc}. The symmetric sesquilinear form $(\Mmin \ \Phi, \Psi )_{\Le} $ is defined for all $\Phi,\Psi \in H_0 (\curlm,\Om)^2$,  and in the case $\Phi = \Psi = \{\Ebf,\Hbf\}$ generates the quadratic form 
$2 \im  (\nabla \times \Ebf , \Hbf)_\Om = - 2 \im  (\nabla \times \Hbf , \Ebf)_\Om $, which is not nonnegative and not nonpositive (this can be seen by the comparison of the values of this quadratic form for $\Phi=\{\Ebf , \Hbf\}$ and $\Phi = \{\Hbf , \Ebf\}$). Under assumption \eqref{e:iZsa},  $M_{\imp,Z}$ is a selfadjoint extension of $M_{0,0}$. 
Thus, $\si (M_{\imp,Z}) \cap \RR_\pm \neq \{\varnothing\}$. Theorem \ref{t:DiscSp} implies that  $\si (M_{\imp,Z}) $ consist of eigenvalues.
Corollary \ref{c:conv} (ii) concludes the proof.
\end{proof}

\subsection{Quantification of nonunique continuation via extremal real eigenvalues}
\label{ss:uc}

For a bounded operator $T:X\to X$ in a Hilbert space $X$, we denote by 
$\re T = \frac12 (T+T^*)$ its real part, and by $\im T = \frac1{2i} (T-T^*)$ its imaginary part.  
Assume that, additionally to \eqref{e:Z1}-\eqref{e:Z2}, the boundary condition 
$   \n   \times \Ebf   = Z  \Hbf_\Tan $ satisfies the assumption that
\begin{gather} \label{e:Zsec}
\text{the real part $\re Z$ of the impedance operator $Z$ is injective (i.e., $\ker (\re Z) = \{0\}$).}
\end{gather}

\begin{prop} \label{p:real}
Suppose \eqref{e:Z1}, \eqref{e:Z2}, and \eqref{e:Zsec}. Let  $\om \in \RR$ be an eigenvalue of $M_{\imp,Z}$ and let $\{\Ebf , \Hbf \}$ be the corresponding eigenfield. Then $\{\Ebf , \Hbf \}$ is a nontrivial solution of the eigenproblem  
 \begin{equation}\label{e:EPM00}
     \ii  \nabla \times \Hbf  = \omega \ve \Ebf, \qquad -\ii  \nabla \times \Ebf = \omega \mu \Hbf , \qquad 0 =\n \times \Ebf,  \qquad 0  = \Hbf_\Tan.
     \end{equation}
\end{prop}     

The proof of this proposition is given in Section \ref{s:real}. It is based on the decomposition of $M_{\imp,Z}$ into selfadjoint and completely nonselfadjoint parts (see \cite{SFBK10} and Proposition \ref{p:mDRed} below). 

In the case of the Leontovich boundary condition $\n \times \Ebf = a \Hbf_\Tan$,  \eqref{e:Zsec} is equivalent to 
\begin{gather} \label{e:Rea}
\text{$\re a(x)>0$ a.e. on $\pa \Om$.}
\end{gather}

Under the Lipschitz continuity, we always mean the uniform Lipschitz continuity. 

\begin{corollary} \label{c:NoReala}
Assume that the impedance coefficient $a(\cdot)$ satisfies \eqref{e:aLinf}, \eqref{e:aCo}, and \eqref{e:Rea}. 
Assume that the material parameters $\ve$ and $\mu$ are piecewise Lipschitz in the following sense:
\begin{gather}\label{e:Lip11}
\text{there exist $N \in \NN$ and open sets $\Om_n \subseteq \Om$, $n=1,\dots,N$, such that $  \overline \Om  = \bigcup_{n =1}^N \overline{\Om}_n$,}
 \\
\text{the boundaries $\pa \Om_n$ are Lipschitz continuous,} 
\label{e:Lip12}
\\
\text{and the restrictions $\ve |_{\Om_n}$ and $\mu |_{\Om_n}$ are Lipschitz continuous in $\Om_n$ for all $n$.} \label{e:Lip2}
\end{gather}
Then $\RR \cap \si (M_{\imp,a}) = \{0\}$ for the Maxwell-Leontovich operator $M_{\imp,a}$.
\end{corollary}

This corollary follows from a more general result of Corollary \ref{c:NoReal}, which is  proved in Section \ref{s:real} with the use of the unique continuation results of \cite{BCT12} and Theorem \ref{t:DiscSp}.

\begin{remark} \label{r:realEig}
The statement that $\{\Ebf , \Hbf \}$ is a nontrivial solution to eigenproblem  \eqref{e:EPM00} is equivalent to the statement that $\{\Ebf , \Hbf \}$ is an eigenfield corresponding to an eigenvalue $\om$ of the symmetric Maxwell operator $M^{\ve,\mu}_{0,0}$ of Section \ref{ss:Disc}. Operators $M^{\ve,\mu}_{\imp,Z}$ are extensions of $M^{\ve,\mu}_{0,0}$. So Proposition \ref{p:real} states that, under the additional assumption \eqref{e:Zsec} on $Z$, a real number $\om$ is an eigenvalue of $\si(M^{\ve,\mu}_{\imp,Z})$ if and only if it is an eigenvalue of $M^{\ve,\mu}_{0,0}$.
\end{remark}

If only the Hölder $C^\ga$-regularity with $\ga<1$ is required from the material parameters  $\ve$ and 
$\mu$, there are counterexamples to the unique continuation property for Maxwell systems
\cite{D12}. 

\begin{remark} \label{r:nonUC}
The results of Demchenko \cite{D12} and the scaling properties of the Maxwell system imply that for every Lipschitz domain $\Om$ and every $\al>0$ there exists $\beta>\al$ such that, 
\begin{equation} \label{e:betaNUC}
\text{for certain $\vep,\mu \in L^\infty (\Om, \MM^\sym_{\al, \beta})$, eigenproblem \eqref{e:EPM00} has at least one nonzero real  eigenvalue.}
\end{equation} 
It is obvious from the structure of the Maxwell system that in this case $(-\om)$ is also an eigenvalue of \eqref{e:EPM00}. So one can assume without loss of generality that $\om>0$ in \eqref{e:betaNUC}.
Let 
\begin{gather} \label{e:betamin}
\text{
$\beta_{\min} (\al,\Om)$ be the infimum of all $\beta>\alpha$ such that  \eqref{e:betaNUC} holds.}
\end{gather}
\end{remark}

\begin{corollary}[extremal eigenvalue for nonunique continuation] \label{c:M00Eig}
Assume that $\al_0>0$ and $\beta_0> \beta_{\min} (\al_0,\Om)$ with $\beta_{\min}$ defined by \eqref{e:betamin}. 
Let us denote $\om_* = \om_* (\al_0,\beta_0,\Om)$ by
\[
\om_* = \inf \{\om>0 : \ \om \ \text{ is an eigenvalue of \eqref{e:EPM00}  for a certain $\{\ve,\mu\} \in L^\infty (\Om, \MM^\sym_{\al_0, \beta_0})^2$} \ \}.
\]
Then $\om_*>0$, and $\om_*$ is an achievable eigenvalue for eigenproblem  \eqref{e:EPM00}
over $\FF = L^\infty (\Om, \MM^\sym_{\al_0, \beta_0})^2$ (in the sense that there exists $\{\ve,\mu\} \in L^\infty (\Om, \MM^\sym_{\al_0, \beta_0})^2$
such that $\om_*$ is an eigenvalue of \eqref{e:EPM00}).
\end{corollary}

The proof of this result is given in Section \ref{s:real}.  This corollary and its proof can be easily adapted 
to an arbitrary H-closed family $\FF \subset L^\infty (\Om, \MM^\sym_{\al_0, \beta_0})^2$ satisfying  the property that there exist at least one achievable over $\FF$ eigenvalue $\om \in \RR \setminus \{0\}$ of eigenproblem \eqref{e:EPM00}.


\section{Main technical tools}
\label{s:Tools}

\subsection{Hodge decompositions and compact embeddings}
\label{ss:Hodge}

This section provides the main technical tools for the proofs in Section \ref{s:Proofs}.
Let  $\xi  \in \Mabsym $. For $X \subseteq L^2 (\Omega,\CC^3)$, we denote $\xi X  := \{ \xi (\cdot) \vbf (\cdot) \ : \ \vbf  \in X \}$ and $\xi^{-1} X  := \{ (\xi (\cdot))^{-1} \vbf (\cdot) \ : \ \vbf  \in X \}$. 
Recall that  $ L_\xi^2 (\Omega,\CC^3) $ is the 'weighted' $L^2$-space defined in Section \ref{ss:Disc}.  
Let $\curlm_1:H^1 (\Om,\CC^3) \to L^2 (\Omega,\CC^3)$ be the continuous operator defined by $\curlm_1 : \vbf \mapsto \nabla \times \vbf$.  Let 
$H^1 (\Om,\CC^3) \ominus \ker \curlm_1$ \ be the orthogonal complement
of the kernel  \ 
$ \ker \curlm_1:= \{\vbf \in H^1 (\Om,\CC^3) : \nabla \times  \vbf = 0 \} $.

The range \quad $\ran \curlm_1 =  \curlm H^1 (\Om,\CC^3) = \{\nabla \times \vbf  \ : \ \vbf \in H^1 (\Om,\CC^3)  \}
$ \quad  of $\curlm_1$  is a closed subspace of $L^2 (\Omega,\CC^3)$ \cite{GR86,C96} (see also Appendix \ref{a:HdecOm}).
Let  $\KK_0 (\pa \Om)$ be the finite-dimensional space of locally constant $\CC$-valued functions  on $\pa \Om$.
The  finite-dimensional space 
\[ \HH_2 ( \Om,\xi) := 
 \{ \wbf = \nabla q \; : \ q\in H^1(\Om), \ 
 \nabla\cdot (\xi \nabla q)=0 \mbox{ in } \Om, \text{ and }
 q |_\paOm \in \KK_0 (\pa \Om) \}
\]
is a `weighted' analogue of the cohomology space in \cite[formula (104)]{C96}.

The `weighted' Helmholtz-Hodge decomposition of the following theorem is a more detailed version of one  particular case of \cite[Corollary 3.3]{PS22} (see also  \cite[Section 8.1]{ACL18}).
\begin{theorem}[cf. \cite{ACL18,PS22}] \label{t:HDecOm}
(i) The Hilbert space $L_\xi^2 (\Omega,\CC^3) $ admits the following orthogonal decomposition 
$L_\xi^2 (\Omega,\CC^3) \ = \ \grad H^1_0 (\Omega) \oplus \HH_2 (\Om, \xi) \oplus \xi^{-1} \curlm H^1 (\Om,\CC^3) $.
\item[(ii)] For every $\ubf \in L_\xi^2 (\Omega,\CC^3) $, there exist a unique triple
\begin{gather} \notag
\text{$p \in H^1_0 (\Omega)$, $\wbf \in \HH_2 ( \Om,\xi)$, $\vbf  \in H^1 (\Om,\CC^3) \ominus \ker \curlm_1$ \quad  such that \quad 
$\ubf = \nabla p + \wbf + \xi^{-1} \curlm \vbf ;$}
\\
\text{moreover, } \qquad \qquad
 \| p \|_{H^1 (\Omega)} + \| \wbf \|_{ L_\xi^2 (\Omega,\CC^3) } + \| \vbf \|_{H^1 (\Om,\CC^3) } \lesssim \| \ubf\|_{L_\xi^2 (\Omega,\CC^3) } \ .
\label{e:HodgeOm3}
\end{gather}
\end{theorem}

For the convenience of the reader, the proof of Theorem \ref{t:HDecOm} is given in Appendix \ref{a:HdecOm}. 

Statement (i) of Theorem \ref{t:HDecOm} can be seen as a more detailed version of the Helmholtz-Weyl-type decomposition  \eqref{e:HWxi}. Indeed, Theorem \ref{t:HDecOm} (i) implies
\begin{equation}\label{e:Div0Dec}
\text{the $L_\xi^2 (\Omega,\CC^3)$-orthogonal decomposition} \quad H(\Div \xi 0,\Om) = \HH_2 (\Om, \xi) \oplus \xi^{-1} \curlm H^1 (\Om,\CC^3) .
\end{equation}

Note that the space $\KK_0 (\pa \Om)$ can be equivalently defined as the kernel of the Laplace-Beltrami operator 
$\De^{\pa \Om} : H^1 (\pa \Om) \to H^{-1} (\pa \Om)$.  For $s \in (0,1]$, the Sobolev-type boundary spaces $H^s (\pa \Om) = H^s (\pa \Om,\CC) $ are well-defined.
Introducing the Hilbert factor-space
$H^0_{\pa \Om}:= L^2 (\pa \Om) / \KK_0 (\pa \Om)$ and, for   $s \in (0,1]$, the factor-spaces  $H^s_{\pa \Om}:= H^s (\pa \Om) / \KK_0 (\pa \Om)$,
one can define the Hilbert space $H^{-s}_{\pa \Om}$ as the dual to $H^s_{\pa \Om}$ w.r.t. the pivot space $H^0_{\pa \Om}$. Note that $H^1_{\pa \Om} = H^1 (\pa \Om) / \KK_0 (\pa \Om)$
can be naturally identified with the orthogonal complement $H^1 (\pa \Om) \ominus \KK_0 (\pa \Om)$ of $\KK_0 (\pa \Om)$.

On $\pa \Om$, we consider the operator of tangential gradient $\grad_{\pa \Om}$ and the  surface vector  curl-operator  $\curlm_{\pa \Om} $ (for their definitions on various Sobolev spaces on $\pa \Om$, we refer to \cite{C96,BCS02,M04,EK22}). Since $\KK_0 (\pa \Om) = \ker \grad_{\pa \Om} = \ker \curlm_{\pa \Om}$, the operators $\grad_{\pa \Om}$ and   $\curlm_{\pa \Om} $ can be considered also as operators defined on the factor-spaces 
$H^s_{\pa \Om}$ with  $s \in [0,1]$ (see \cite{EK22}). The quadratic form 
$  (f|f)_1 = \| \grad_{\pa \Om} f \|^2_\paOm $ defined for $ f \in H^1_{\pa \Om} $ is closed in $H^0_{\pa \Om}$ (in the sense of \cite{Kato}). Using the representation theorems \cite[Theorems VI.2.1 and VI.2.23]{Kato}, one can define another version of the Laplace-Beltrami operator as a unique nonpositive selfadjoint operator $\De_{\pa \Om}$ in $H^0_{\pa \Om}$ such that 
$H^1_{\pa \Om} = D \left(|\De_{\pa \Om}|^{1/2}\right)$ and $(-\De_{\pa \Om} f|f)_{H^0_{\pa \Om}} =  (f|f)_1$ for all 
$f \in D(\De_{\pa \Om})$.
The operator $\De_{\pa \Om}$ is boundedly invertible, and there exists 
an orthonormal basis $\{u_k\}_{k=1}^{\infty}$ in $H^0_{\pa \Om}$ such that
$(-\De_{\pa \Om}) u_k = \la_k^2 u_k$ and $\si (\De_{\pa \Om}) = \{-\la_k^2 \}_{k=1}^{\infty}$, \  where  \ $\la_k >0$ for all \ $k \in \NN$.

We use the scale $H^s_{\De_{\pa \Om}}$, $s \in \RR$, of Hilbert spaces associated with $|\De_{\pa \Om}|^{1/2}$. In particular, for $s \ge 0$, we define  the Hilbert space $H^s_{\De_{\pa \Om}}$ as $D(|\De_\paOm|^{s/2} )$ equipped  with the graph norm, and define $H^{-s}_{\De_{\pa \Om}}$ as the dual to $H^s_{\De_{\pa \Om}}$ w.r.t. the pivot space $H^0_{\De_{\pa \Om}} := H^0_{\pa \Om}$.
For $s \in [-1,1]$, it follows  from \cite{GMMM11} that
 $H^s_{\De_\paOm}$ can be identified with $H^s_{\pa \Om}$ up to equivalence of the norms.
 
The orthogonal Hodge decomposition   of the space $\LLt$ 
can be written as 
\begin{gather} \label{e:HDL22}
\LLt =  \grad_{\pa \Om} H^1_{\De_{\pa \Om}} \oplus \KK_1  (\pa \Om) \oplus \curlm_{\pa \Om} H^1_{\De_{\pa \Om}} ,
\end{gather}
where  $\KK_1 (\pa \Om) := \{ \ubf \in \LLt \; : \; 0=\Div_{\pa \Om } \ubf=\curl_{\pa \Om} \ubf  \}$
is the finite-dimensional cohomology space of $\pa \Om $ \cite{C96,BCS02,BHPS03,EK22}.

The Hodge decompositions of the trace spaces $\HH^{-1/2} (\curl_{\pa \Om}, \pa \Om) = \pi_\top \HH (\curlm,\Om)$ and \linebreak $\HH^{-1/2} (\Div_{\pa \Om}, \pa \Om) = \n^\times  \HH (\curlm,\Om) $
 can be written as the following direct sums of closed subspaces \cite{BCS02,BHPS03,M04} (see also \cite{C96} and \cite[Section 4]{EK22}):
\begin{gather} \label{e:HDcurl}
\HH^{-1/2} (\curl_{\pa \Om},\pa \Om) = \grad_{\pa \Om} H^{1/2} _{\De_{\pa \Om}} \dot + \KK_1  (\pa \Om) \dot + \curlm_{\pa \Om} H^{3/2}_{\De_{\pa \Om}}  ,  \\
\HH^{-1/2} (\Div_{\pa \Om}, \pa \Om) = \grad_{\pa \Om} H^{3/2}_{\De_{\pa \Om}} \dot + \KK_1  (\pa \Om) \dot + \curlm_{\pa \Om} H^{1/2} _{\De_{\pa \Om}}   .
\label{e:HDdiv}
\end{gather}

For the Hilbert space 
\begin{gather}
 X_\imp (\curlm,\Div \xi, \Om) := \{ \ubf \in H_\imp (\curlm,\Om) \; : \; \nabla\cdot(\xi \ubf )\in L^2(\Omega) \} \quad \text{(see \cite{ACL18,PS22}),} 
 \label{e:Ximp}
 \end{gather}
 with the graph norm  defined by
$
\| \ubf \|_{X_\imp (\curlm,\Div \xi, \Om)}^2 =  \|\ubf\|^2_{H (\curlm,\Om) } + \|\nabla\cdot (\xi \ubf)\|^2_\Om+\| \n \times  \ubf\|^2_\paOm \;,
$
\begin{gather} \label{e:XimpCompImb}
\text{the compact embedding } \qquad  X_\imp (\curlm,\Div \xi, \Om) \imb \imb L^2(\Omega,\C^3) \quad \text{ holds.}
\end{gather}
This embedding  is a particular case of \cite[Theorem 8.1.3]{ACL18} (see also  \cite[Theorem 4.1]{PS22}).

\subsection{Elements of the spectral theory of m-dissipative operators}
\label{ss:SpTh}

This subsection collects the main definitions and facts of spectral theory of nonselfadjoint operators used in this paper. 
Let $X$ be a Hilbert space and let $T:D(T) \subseteq X \to X$ be a (linear) operator in $X$.
An eigenvalue $\om$ of $T$ is called \emph{isolated} if $\om$ is an isolated point of the spectrum $\si (T)$ of $T$. The \emph{discrete spectrum} $\si_\disc (T)$ of $T$ is the set of isolated eigenvalues of $T$ with finite algebraic multiplicities (see \cite{Kato,RSIV}). We say that $T$ has  \emph{purely discrete spectrum} if $\si(T) = \si_\disc (T)$. The closed set $\si_\ess (T) = \si (T) \setminus \si_\disc (T)$ is called an \emph{essential spectrum} of $T$ \cite{RSIV}.

Assume that $\om_0$ belongs to the resolvent set $\rho (T)$ of $T$ and the resolvent $(T-\om_0 )^{-1} = (T-\om_0 I_X)^{-1}$ at $\om_0$ is a compact operator (for brevity, we sometimes omit the identity operator $I_X$ in $X$ from the resolvent-type notations). Then the resolvent $(T-\om)^{-1}$ is compact for every $\om \in \rho (T)$; in this case, it is said that $T$ is an \emph{operator with compact resolvent} (this definition assumes $\rho (T) \neq \varnothing$). An operator $T$ with compact resolvent has purely discrete spectrum \cite[Theorem III.6.29]{Kato}.

A closed subspace $X_1$ of $X$ is called an \emph{invariant subspace} of $T$ if $Tf \in X_1$ for every $f \in D(T) \cap X_1$. In this case, the operator $T$ restricted to 
$D (T |_{X_1}) := D(T) \cap X_1 $ generates in the Hilbert subspace $X_1$ an operator $T |_{X_1} : D (T |_{X_1})  \subset X_1 \to X_1$, which we call the part of $T$ in $X_1$.  

If $X_1$ and $X_2$ are invariant subspaces of $T$ such that the orthogonal decomposition $X = X_1 \oplus X_2$ takes place, the subspaces $X_1$ and $X_2$ are called \emph{reducing subspaces} of $T$, and one says that the decomposition $X = X_1 \oplus X_2$ reduces $T$ to the orthogonal sum of its parts $T = T  |_{X_1} \oplus T  |_{X_2}$ (e.g., in Section \ref{ss:Disc}, the orthogonal decomposition 
$\Le = \Sem \oplus \GG_{0} (\Om)$ reduces the symmetric Maxwell operator $M_{0,0}$, and as a consequence, reduces its m-dissipative extension $M_{\imp,Z}$ \cite{EK22}).

An operator $T$ is called \emph{completely nonselfadjoint} if the zero space $\{0\}$ 
is the only reducing subspace $X_1$ of $T$ with the property $T |_{X_1} = (T |_{X_1})^*$. 
A part $T|_{X_1}$ of an operator $T$ in a reducing subspace $X_1$ is called a \emph{maximal selfadjoint part of the operator $T$} if the part $T |_{X_2}$ of $T$ in the orthogonal complement $X_2  = X \ominus X_1$ is completely nonselfadjoint.

Accretive, dissipative, and m-dissipative operators were defined in the beginning of Section \ref{ss:Disc}. 
The following statements are equivalent: (a) $T$ is m-dissipative, (b) $(-\ii)T$ is a generator of a contraction semigroup \cite{P59,E12,Kato},  
 (c) $\ii T$ is a closed maximal accretive operator \cite{P59,Kato}, (d) $\ii T$ is a densely defined maximal accretive  operator \cite{P59,SFBK10}.
The following proposition follows immediately from the combination of these equivalences with the results of \cite[Section 4.4]{SFBK10}.

\begin{prop}[\cite{SFBK10}] \label{p:mDRed}
Let $T$ be an m-dissipative operator in $X$. 
Then:
\item[(i)] There exists a unique  orthogonal decomposition $X = X_\sa \oplus X_\nsa$ 
that reduces $T$ to an orthogonal sum $T = T |_{X_\sa} \oplus T |_{X_\nsa}$ of the maximal selfadjoint part  $T_\sa = T |_{X_\sa}$  of $T$ and a completely nonselfadjoint m-dissipative operator $T_\nsa =  T |_{X_\nsa} $.
\item[(ii)] Assume that $f$ is an eigenvector of $T$ associated with a real eigenvalue $\om$.
Then $f \in D(T_\sa ) \subseteq D(T^*)$ and $T^* f = \om f$ (i.e., $f$ is also  an eigenvector associated with the eigenvalue $\om$ of $T^*$).
\end{prop}


\section{Proofs of main results}
\label{s:Proofs}

\subsection{Proofs of Theorems \ref{t:b} and \ref{t:DiscSp}, discreteness of spectra}
\label{s:ProofsDisc}

Assume that $\E \in \Himp$ and $\H \in \Himp$. Then  
\begin{gather*}
\H_\Tan \in \LLt \cap \pi_\top H (\curlm,\Om) = \LLt \cap H^{-1/2} (\curl_{\pa \Om}, \pa \Om) , 
\\
\n\times \E \in \LLt \cap \n^\times H (\curlm,\Om) = \LLt \cap H^{-1/2} (\Div_{\pa \Om}, \pa \Om) .
\end{gather*}
From the Hodge decompositions \eqref{e:HDL22}, \eqref{e:HDcurl}, and \eqref{e:HDdiv} 
one obtains that 
 \begin{gather} \label{e:paOmHodge}
  \H_\Tan = \grad_\paOm  p_1 + \vbf_1 +
  \curlm_\paOm q_1 \quad \text{ and } \quad \n\times \E  = \grad_\paOm  p_2 + \vbf_2 +
  \curlm_\paOm q_2
   \end{gather}
with certain uniquely determined  $p_1,q_2 \in H^1_\paOm $, $q_1,p_2 \in  H^{3/2}_{\De_\paOm}$, and $\vbf_1, \vbf_2 \in \KK_1 (\pa \Om) $, and that
for $\{\E,\H\} \in \Himp^2$,
\begin{gather}
\| \H_\Tan \|^2_\paOm \lesssim   \|p_1\|^2_{H^1_\paOm} +  \| \vbf_1 \|_\paOm^2 + \|q_1\|_{H^1_\paOm}^2    \label{e:impH1}
 ,\\
 \qquad  \|p_1\|^2_{H^{1/2}_\paOm} + \| \vbf_1 \|_\paOm^2 +  \|q_1\|_{H^{3/2}_{\De_\paOm}}^2
 \lesssim  \| \H_\Tan \|^2_{H^{-1/2} (\curl_{\pa \Om}, \pa \Om)} \lesssim \| \H \|^2_{H(\curlm,\Omega)}, \label{e:impH2} \\
\text{ and }    \|p_2\|^2_{H^{3/2}_\paOm} +  \| \vbf_2 \|_\paOm^2 + \|q_2 \|_{H^{1/2}_\paOm}^2 \lesssim  \| \n \times \E  \|^2_{H^{-1/2} (\Div_{\pa \Om}, \pa \Om)} \lesssim \| \E \|^2_{H(\curlm,\Omega)} .
   \label{e:impH3}
\end{gather}

The impedance operator $Z$ satisfies the assumptions \eqref{e:Z1}-\eqref{e:Z2}. It follows from the Lax-Milgram lemma that $Z:\LLt \to \LLt$ is a homeomorphism (see, e.g., \cite[Section 4.2]{ACL18}).

\begin{proof}[Proof of Theorem \ref{t:b}.]
 Statement (ii) of the theorem follows immediately from the combination of statement (i), statement \eqref{e:NDeqHcurl2}, and formula \eqref{e:NHimp2}.  We prove statement (i) in two steps.

\emph{Step 1.} The assumption $\{\E,\H\} \in D(M_{\imp,Z})$ is equivalent to $\{\E,\H\} \in \Himp^2$ and $\n \times \E   = Z  \H_\Tan$.
In particular, the vector-fields $\n \times \E$, $\E_\Tan$, and $\H_\Tan$ belong to $\LLt$.

Due to $ \| \n \times \ubf\|_{\pa \Om} = \| \ubf_\Tan \|_{\pa \Om} $ and the fact that the operator $Z$ is bounded in $\LLt$,
\begin{equation*} \label{e:E<H<E}
 \| \E_\Tan \|_\paOm = \|\n \times \E \|_\paOm \lesssim \| \H_\Tan\|_\paOm 
 \quad \text{ for all } \{\E,\H\} \in D(M_{\imp,Z}) .
\end{equation*}
Hence, in order to prove statement (i) of Theorem \ref{t:b}, it is enough to prove 
\begin{gather}  \label{e:Dimb<2}
    \|\H_{\Tan}\|^2_{\partial \Omega} \lesssim 
  \|\E\|^2_{H(\curl,\Omega)} + 
  \|\H\|^2_{H(\curl,\Omega)} 
 \qquad \text{  for all $\{\E,\H\} \in D(M_{\imp,Z})$ }.
 \end{gather}

\emph{Step 2.}
Using the 1st decomposition in  \eqref{e:paOmHodge}, we put $g :=  \Div_\paOm \left( Z \grad_\paOm p_1 \right) \in H^{-1}_\paOm $. From  \eqref{e:paOmHodge}  and the boundary condition $\n\times \E = Z \H_\mathrm{tan} $,
one gets
\[
 \begin{split}
 g = \Div_\paOm \left( Z \grad_\paOm p_1\right) &=
 \Div_\paOm ( Z \H_\Tan)   -\Div_\paOm (Z \vbf_1) -
  \Div_\paOm   \left(Z \curlm_\paOm q_1\right)
 \\&=
 \Div_\paOm ( \n\times E) - 
 \Div_\paOm (Z \vbf_1) -
  \Div_\paOm   \left(Z \curlm_\paOm q_1\right) \;.
  \end{split}
\]
The assumptions \eqref{e:Z1}-\eqref{e:Z2} imply that 
$( Z \mathbf{grad}_{\partial\Omega} \wbf,\mathbf{grad}_{\partial\Omega} \wbf)_{\partial\Omega}$ is a coercive sesquilinear form on $H^1_{\pa \Om}$ (in the sense of \cite[Remark 4.2.7]{ACL18}).

Using the Lax-Milgram lemma (e.g., in the form of \cite[Theorem 4.2.8]{ACL18}) and the 2nd decomposition in  \eqref{e:paOmHodge}, we have the estimate
\begin{multline*} 
 \|p_1\|_{H^1_{\pa \Om}} \lesssim \|g\|_{H^{-1} (\pa \Om)}  
\lesssim  \|\Div_\paOm  \grad_\paOm p_2 \|_{H^{-1} (\pa \Om)} +\| \Div_\paOm (Z \vbf_1) \|_{H^{-1} (\pa \Om)} 
\\ + \| \Div_\paOm   \left(Z \curlm_\paOm q_1\right) \|_{H^{-1} (\pa \Om)} \ ,
\end{multline*}
where we used the facts that $\vbf_2 \in \KK_1 (\pa \Om)$, $\Div_\paOm  \vbf_2 = 0$,
and $\Div_\paOm  \curlm_\paOm q_2 = 0$ (see \cite{BCS02,EK22}).
This implies 
\quad $\|p_1\|_{H^1_{\pa \Om}}
\lesssim \|\Div_\paOm  \grad_\paOm p_2 \|_{H^{-1} (\pa \Om)} + \| Z \vbf_1 \|_{\paOm}  + 
 \|Z \curlm_\paOm q_1\|_\paOm \;. $

Taking into account the facts that $Z$ is bounded in $\LLt$ and
$\Div_\paOm  \grad_\paOm : H^1_{\pa \Om} \to H^{-1} (\pa \Om)$ is a homeomorphism (which can be identified with the Laplace-Beltrami operator generated by $\De^{\pa \Om} $ on the factor-space $H^1_{\pa \Om}$),
one obtains 
\begin{gather}
 \|p_1\|_{H^1_{\pa \Om}} \lesssim \| p_2 \|_{H^1_\paOm} + \| \vbf_1 \|_{\paOm}  + 
 \|q_1\|_{H^1_\paOm} .
 \label{e:p1<3}
\end{gather}

From \eqref{e:impH3}, we see that  $\| p_2 \|_{H^1_\paOm} \lesssim \| p_2 \|_{H^{3/2}_\paOm} \lesssim \| \E \|^2_{H(\curlm,\Omega)}$. This,  \eqref{e:impH2}, and  \eqref{e:p1<3} imply 
\[
 \|p_1\|_{H^1_{\pa \Om}} \lesssim \| \E \|_{H(\curlm,\Omega)} +  \| \vbf_1 \|_{\paOm}  + 
 \|q_1\|_{H^1_\paOm} \lesssim \| \E \|_{H(\curlm,\Omega)} + \| \H \|_{H(\curlm,\Omega)}.
\] 
Applying this estimate to  \eqref{e:impH1} and combining again  with \eqref{e:impH2}, we obtain finally 
\[
\| \H_\Tan \|^2_\paOm \lesssim   \|p_1\|^2_{H^1_\paOm} +  \| \vbf_1 \|_\paOm^2 + \|q_1\|_{H^1_\paOm}^2  \lesssim  \| \E \|^2_{H(\curlm,\Omega)} + \| \H \|^2_{H(\curlm,\Omega)} ,
\]
which proves \eqref {e:Dimb<2} of Step 1, and so, completes the proof of Theorem \ref{t:b}.
\end{proof}

In order to prove Theorem \ref{t:DiscSp}, we show that the restricted Maxwell operator $M_{\imp,Z} |_{\Sem} $ has compact resolvent. 
This result follows from the combination of Theorem \ref{t:b} with the compact embedding \eqref{e:XimpCompImb}, more precisely, with the two compact embeddings 
\begin{gather} \label{e:Ta}
 X_{\imp} (\curlm,\Div \ve, \Om) \imb \imb L^2(\Omega,\C^3) \, ,  \qquad  X_{\imp} (\curlm,\Div \mu, \Om) \imb \imb L^2(\Omega,\C^3) .
\end{gather}

\begin{proof}[Proof of Theorem \ref{t:DiscSp}]
Let us prove statement (i). As before, the domains $D(M_{\imp,Z} |_\Sem)$ and $D(M_{\imp,Z}) $ are assumed to be equipped with the graph norm of 
$M_{\imp,Z}$.

By \cite[Remark 2.2]{EK22}, the restricted Maxwell operator $M_{\imp,Z} |_\Sem$ is m-dissipative. Hence, for all $\om \in \CC_+$, the resolvent $(M_{\imp,Z} |_\Sem -\om )^{-1}:\Sem \to \Sem$ of $M_{\imp,Z}$ is a continuous operator, and so, the operator $M_{\imp,Z} |_\Sem -\om I_\Sem $ is a bijective bounded operator from the Hilbert space $D(M_{\imp,Z} |_\Sem)$ onto $\Sem$. By the bounded inverse theorem, the resolvent $(M_{\imp,Z} |_\Sem -\om )^{-1}$ can be considered as 
a linear homeomorphism from 
$\Sem$ onto  $D(M_{\imp,Z} |_\Sem)$.

 By Theorem \ref{t:b},
the continuous embedding  $ D(M_{\imp,Z}) \ \imb \ \Himp^2 $ holds. 
The assumptions $ \Div (\ve E) = 0 $  and $\Div (\mu H) = 0 $ (which define the subspace $\Sem$ of $\Le$)
imply that on $D(M_{\imp,Z} |_\Sem)$ the norms of the spaces $ X_\imp (\curlm,\Div \ve, \Om) \oplus  X_\imp (\curlm,\Div \mu, \Om)$ and $\Himp^2 $ coincide.
This yields  the continuous embedding 
\begin{gather*} 
D(M_{\imp,Z} |_\Sem)  \quad \imb \quad  X_\imp (\curlm,\Div \ve, \Om) \oplus  X_\imp (\curlm,\Div \mu, \Om).
\end{gather*} 

Combining this continuous embedding with the compact embeddings \eqref{e:Ta} and the equivalence of norms \eqref{e:Le0}-\eqref{e:Le2},
we see that the resolvent $(M_{\imp,Z} |_\Sem -\om )^{-1}$ is compact as an operator from $\Sem$ to $\Le$. Since the range \ $\ran (M_{\imp,Z} |_\Sem -\om )^{-1}$ \  lies in $\Sem$, the operator
$(M_{\imp,Z} |_\Sem -\om )^{-1}$ is compact as an operator in $\Sem$. Thus, by \cite[Theorem III.6.29]{Kato},  $M_{\imp,Z} |_\Sem $ has purely discrete spectrum. 

Now statement (ii) of Theorem \ref{t:DiscSp} follows immediately from statement (i) and the decomposition $M_{\imp,Z} = M_{\imp,Z} |_\Sem \oplus 0$ (see Section \ref{ss:Disc}). 
This completes the proof of Theorem \ref{t:DiscSp}.
\end{proof}

\subsection{Proofs of Theorems \ref{t:c} and \ref{t:EigFree}, convergence of eigenvalues}\label{s:iii}


Let $Z:\LLt \to \LLt$ be an impedance operator satisfying \eqref{e:Z1}-\eqref{e:Z2}.
Recall that $\Mmin$ is a symmetric Maxwell operator with the domain $H_0 (\curlm,\Om)^2$
and that $\Mmin^*$ is the Maxwell operator with the  domain $H (\curlm,\Om)^2$.
The m-dissipative Maxwell operator $M_{\imp,Z} = M_{\imp,Z}^{\ve,\mu}$ is  a restriction $\Mmin^*$ defined by the boundary condition $   \n   \times \Ebf   = Z  \Hbf_\Tan $ (and is an extension of $\Mmin$).

Statement (iii) of the next proposition provides the weak formulation for the eigenproblem. 

\begin{prop}
\label{p:weak}
(i) The adjoint $M^*_{\imp,Z}$ of $M_{\imp,Z}$ in $\Le$ is the Maxwell operator with the same material parameters $\ve,\mu \in \Mabsym$, but  associated with the boundary condition \linebreak 
$ \n \times \E   = - Z^*  \H_\Tan $. That is,
$M^*_{\imp,Z}$ is defined by the differential expression $\Mf (\ve,\mu)$ on the domain 
\begin{equation}\label{e:DM*}
D (M^*_{\imp,Z}) = \{ \{\E,\H\} \in  H_\imp (\curlm,\Om)^2 \ : \ \n \times \E   = - Z^*  \H_\Tan \} .
\end{equation}

\item[(ii)] The operator $(-1) M^*_{\imp,Z}$ is m-dissipative in $\Le$.

\item[(iii)] A pair $\{\E,\H\} \in H (\curlm,\Om)^2 $ is a solution to the problem 
 \begin{equation}\label{15}
     \ii  \nabla \times \H  = \omega \ve \E, \qquad -\ii  \nabla \times \E = \omega \mu \H , \qquad \n \times \E   = Z  \H_\Tan
     \end{equation}
if and only if $\{\E,\H\}$    satisfies the identity
\begin{equation}\label{16}
 (\E, \ii \nabla\times \H_* - \overline\om\ve \E_*)_\Omega +
 (\H,-\ii \nabla\times \E_* - \overline \om \mu \H_*)_\Omega 
 =0 \quad \text{for all $\{\E_*,\H_*\} \in D (M^*_{\imp,Z})$.}
\end{equation}
\end{prop}
\begin{proof}
(i) can be seen directly by the  integration by parts, or can obtained from the combination of \cite[Remark 6.1]{EK22} with \cite[Theorem 8.1 and Definition 7.2]{EK22}. 
(ii) follows from the m-dissipativity of the operator $M_{\imp,Z}$ and \cite[Problem V.3.21]{Kato}. 
(iii) is a reformulation of the well-known statement that $\ker (M_{\imp,Z} -\om ) = \Le \ominus \overline{\ran (M_{\imp,Z}^* - \overline{\om})}$.
\end{proof}

Let $\ve_n , \mu_n  \in \Mabsym$ for all $n \in \NN$. Assume that,
for each $n \in \NN$, the vector-field  $\{\Ebf^n ,\Hbf^n \}$ is an eigenvector of the Maxwell operator $\Mc_n :=  M_{\imp,Z}^{\ve_n,\mu_n}$ associated with a certain eigenvalue $\om_n \in \si (\Mc_n)$,
i.e., $\{\Ebf^n,\Hbf^n \}$ is a nontrivial solution to the eigenproblem 
\begin{gather}  \label{e:EigEqn}
  \quad \ii  \nabla \times \Hbf^n  = \omega_n \ve_n \Ebf^n ,  \quad - \ii  \nabla \times \Ebf^n =  \omega_n \mu_n \H^n , \quad 
  \n \times \Ebf^n   = Z  \Hbf_\Tan^n  .
 \end{gather}


\begin{lemma} \label{l:ConvSub1}
Assume additionally that, for each $n \in \NN$, the eigenfield $\{\Ebf^n ,\Hbf^n \}$ is 
\begin{gather}
\text{$\LL^2_{\ve_n,\mu_n} (\Om)$-normalized  in the sense  } \quad (\ve_n \Ebf^n , \Ebf^n)_\Om + (\mu_n \Hbf^n , \Hbf^n)_\Om = 1 , \label{e:EHnorm}
\end{gather} 
and that the sequence of eigenvalues $\{ \om_n \}_{n \in \NN}$ is bounded in $\CC$.
Then there exists a subsequence $\{n_k\}_{k=1}^\infty \subset \NN$ such that the following statements hold:
\item[(C1)] the subsequences $\{\E^{n_k}\}_{k \in \NN}$ and $\{\H^{n_k} \}_{k \in \NN}$ converge weakly in $L^2 (\Omega,\CC^3)$ to certain vector-fields $\E^\infty \in H (\curlm,\Om)  $ and  $\H^\infty \in H (\curlm,\Om)$;

\item[(C2)] $\{\nabla \times \E^{n_k}\}_{k \in \NN}$ and $\{\nabla \times \H^{n_k} \}_{k \in \NN}$ converge weakly in $L^2 (\Omega,\CC^3)$ to 
$\nabla \times \E^\infty $ and, resp., $\nabla \times \H^\infty $;

\item[(C3)] $\{\n \times \E^{n_k}\}_{k \in \NN}$, $\{ \E^{n_k}_\Tan\}_{k \in \NN}$, $\{\n \times \Hbf^{n_k}\}_{k \in \NN}$, and $\{\H^{n_k}_\Tan\}_{k \in \NN}$ converge weakly in $\LLt$ to 
$\n \times \E^\infty $, $\E^\infty_\Tan $, $\n \times \H^\infty $and $ \H_\Tan^\infty $, respectively.

\item[(C4)] $\{\ve_{n_k}\}_{k \in \NN}$ and $\{\mu_{n_k}\}_{k \in \NN}$ H-converge to certain $\ve_\infty \in \Mabsym$ and $\mu_\infty \in \Mabsym$;

\item[(C5)] $\{ \om_{n_k} \}_{k \in \NN}$ converges to a certain $\om_\infty \in \CC$.
\end{lemma}

\begin{proof}
\emph{Step 1. Let us prove (C1)-(C2).} The weak $L^2 (\Omega,\CC^3)$-convergences $\E^{n_k}  \Weak \E^\infty$ and \linebreak $\H^{n_k}  \Weak \H^\infty $ after a passage to certain subsequences  follow from the normalization of \eqref{e:EHnorm}.  Similarly, from \eqref{e:EigEqn}
 and inclusions  $\ve_n , \mu_n \in \Mabsym$, one sees that the sequences 
 $\{\nabla \times \E^n\}_{n \in \NN}$ and $\{\nabla \times \H^n \}_{n \in \NN}$ are bounded in $L^2 (\Omega,\CC^3)$. Hence, after a possible iterative passage to subsequences, we obtain  $ \nabla \times \E^{n_k} \Weak \ubf_1$ and $ \nabla \times \H^{n_k}  \Weak \ubf_2$ in $L^2 (\Omega,\CC^3)$ with certain $\ubf_1, \ubf_2  \in L^2 (\Omega,\CC^3)$. 
 
 In order to finish the proof of (C1)-(C2), one has to show that 
 $\E^\infty, \H^\infty \in H (\curlm,\Om) $, \ $  \nabla \times \E^\infty = \ubf_1$, \ and 
 $ \nabla \times \H^\infty = \ubf_2 $.
  These statements follows from the fact that the operator $\curlm$ is closed,  see \cite[Problem III.5.12]{Kato}.
 
\emph{Step 2. Let us prove (C3).} From the proof of (C1)-(C2), one sees that the sequences $\{ \E^n\}_{n \in \NN}$ and $\{ \H^n \}_{n \in \NN}$ are bounded in $H (\curlm,\Om)$.  We will use the fact that $\vbf \mapsto \n \times \vbf$ is a unitary operator in $\LLt$ (see Section \ref{ss:MS}). 
Since $\n\times \Ebf^n = Z \Hbf_\Tan^n $, Theorem \ref{t:b} (i) implies that $\{ \Ebf_\Tan^n\}_{n \in \NN}$ and $\{ \Hbf_\Tan^n \}_{n \in \NN}$ are bounded in $\LLt$, and so, $\{ \n \times \E\}_{n \in \NN}$ is also bounded in $\LLt$. After passing to suitable subsequences, we get
the weak $\LLt$-convergences $\n \times \Ebf^{n_k}  \Weak \wbf_1$ and $\Hbf_\Tan^{n_k}  \Weak \wbf_2 $ to certain $\wbf_1, \wbf_2 \in \LLt$.
The mappings $\Ebf \mapsto \n \times \Ebf$ and 
$\Hbf \mapsto \Hbf_\Tan$ are bounded and  closed as operators from the Hilbert space $\Himp$ to the Hilbert space $\LLt$. Thus $\wbf_1 = \n \times \E^\infty$
and $\wbf_2 =  \Hbf_\Tan^\infty \in L^2 (\Omega,\CC^3)$. This implies also the rest of (C3) (we use here again that $\vbf \mapsto \n \times \vbf$ is a unitary operator in $\LLt$).

\emph{Step 3.} Conditions (C4)-(C5) are fulfilled after an additional passage to a subsequence due to  the H-compactness of $\Mabsym$ (see Remark \ref{r:Hcon}) and the boundedness of $\{ \om_n \}_{n \in \NN}$.
\end{proof}

\begin{lemma} \label{l:ConvSub2}
Assume that a sequence of $\LL^2_{\ve_n,\mu_n} (\Om)$-normalized eigenfields $\{\{\Ebf^n,\Hbf^n \}\}_{n \in \NN}$ satisfies statements (C1)-(C5) of Lemma \ref{l:ConvSub1} (with the index $n \in \NN$ instead of $n_k$).
Assume that for each $n \in \NN$ at least one of the two following conditions is fulfilled: 
\begin{gather}  \label{e:AsDiv=0}
\text{$\om_n \neq 0$  \qquad or \qquad  $0 = \Div (\ve_n \E^n) = \Div (\mu_n \H^n)$ \ .}
\end{gather}
Then: 
\item[(C6)]  There exists a subsequence $\{n_k\}_{k=1}^\infty \subseteq \NN$ such that
 $\{\ve_{n_k} \Ebf^{n_k}\}_{k \in \NN}$ and $\{ \mu_{n_k} \Hbf^{n_k}\}_{k \in \NN}$ converge weakly in $L^2 (\Omega,\CC^3)$ to $\ve_\infty \E^\infty $ and, resp., $\mu_\infty \H^\infty $.
 \item[(C7)] $(\ve_\infty \Ebf^\infty, \Ebf^\infty )_\Omega  + (\mu_\infty \Hbf^\infty, \Hbf^\infty)_\Omega=1$.
\end{lemma}

\begin{proof}   
In the case $\xi \equiv I_{\RR^3}$, we denote $H(\Div  0,\Om):=  H(\Div \xi 0,\Om)$ and $ \HH_2 ( \Om) := \HH_2 ( \Om, \xi)$.
Let  $\Dbf^n=\ve_n \Ebf^n$ \ and \ $\Bbf^n=\mu_n \Hbf^n$, \  $n \in \NN$.
Note that assumption \eqref{e:AsDiv=0} implies always its second option, i.e., 
\begin{equation} \label{e:0=divD=divB}
\text{$0 = \Div \Dbf^n = \Div \Bbf^n$  for all $n \in \NN$.}
\end{equation} 
Indeed,  if $\om_n \neq 0$, then the equations in \eqref{e:EigEqn} yield \eqref{e:0=divD=divB}.
We split the proof into two steps.

\emph{Step 1. Let us prove (C6)}. From (C1) and  $\ve_n, \mu_n  \in \Mabsym$,  we infer that 
$\{\Dbf^{n}\}_{n \in \NN}$ and $\{ \Bbf^n\}_{k \in \NN}$ are bounded in $L^2 (\Omega,\CC^3)$. So, after possible passing to suitable subsequences, one gets $\Dbf^{n_k} \Weak \Dbf^\infty $ and $\Bbf^{n_k} \Weak \Bbf^\infty$ in $L^2 (\Omega,\CC^3)$. By (C1) and (C4),  $ \nabla \times \Ebf^{n_k} \Weak \Ebf^\infty$ and $ \nabla \times \H^{n_k}  \Weak \Hbf^\infty$ in $L^2 (\Omega,\CC^3)$, as well as, $\ve_{n_k} \Hcon \ve_\infty$ and 
$\mu_{n_k} \Hcon \mu_\infty$. Now, the H-convergence criteria of Proposition \ref{p:Hcon} and Remark \ref{r:HconC} imply  $\Dbf^\infty = \ve_\infty \E^\infty $ and $\Bbf^\infty = \mu_\infty \H^\infty $. This gives (C6).

\emph{Step 2. Let us prove (C7).} It follows from \eqref{e:0=divD=divB}
that $ \Dbf^n \in H(\Div 0,\Om)  $ and $\Bbf^n \in  H(\Div 0,\Om)$ for all $n \in \NN$. We apply the decomposition \eqref{e:Div0Dec} and Theorem \ref{t:HDecOm} with $\xi \equiv I_{\RR^3}$ in order to construct for each $n$ the vector fields 
  $\wbf^n  \in \HH_2 ( \Om)$ and $\vbf^n \in H^1 (\Om,\CC^3) \ominus \ker \curlm_1$ such that
  $    \Dbf^n = \wbf^n + \nabla \times \vbf^n $.
 Using  $\ve_n, \mu_n  \in \Mabsym $, \eqref{e:EHnorm}, 
 and estimate \eqref{e:HodgeOm3}, 
 we see that $\{ \wbf^n\}_{n \in \NN}$  is bounded in the closed subspace $\HH_2 ( \Om)$ of $L^2 (\Omega,\CC^3)$, and $\{ \vbf^n\}_{n \in \NN}$ is bounded in  $H^1 (\Om,\CC^3)$. 
 Now we can pass to subsequences such that $\wbf^{n_k} \to \wbf^\infty$  in $ \HH_2 (\Om)$ (recall that  $\HH_2 ( \Om)= \HH_2 ( \Om, I_{\RR^3})$ is finite-dimensional), 
 $\vbf^{n_k} \Weak \vbf^\infty$  in $H^1 (\Om,\CC^3)$, and such that (C6) is satisfied.
 By Theorem \ref{t:HDecOm}, the operators that recover $\wbf^n$ and $\vbf^n$ from $\Dbf^n $ are continuous. Hence, 
 $
    \Dbf^\infty = \wbf^\infty +  \nabla \times  \vbf^\infty .
    $
    
Combining  $\vbf^{n_k} \Weak \vbf^\infty$  in $H^1 (\Om,\CC^3)$ with the compact embedding 
 $H^1 (\Om,\CC^3) \imb \imb L^2 (\Omega,\CC^3)$, we get the (strong) convergence $\vbf^{n_k} \to \vbf^\infty$ in $L^2 (\Omega,\CC^3)$, and in turn, 
 $\ga(\vbf^{n_k})  \to \ga (\vbf^\infty) $ in $L^2 (\pa \Omega,\CC^3)$ and  $\n \times \vbf^{n_k}  \to \n \times \vbf^\infty $ in $\LLt$, 
 where $\ga : H^1 (\Om,\CC^3) \to H^{1/2} (\pa \Om,\CC^3) $ is the trace operator $ \vbf \mapsto \vbf|_\paOm$. Using the integration by parts, one obtains 
\[
 \begin{split}
 (\ve_{n_k} \Ebf^{n_k},\Ebf^{n_k})_\Omega &=(\nabla \times \vbf^{n_k}, \Ebf^{n_k})_\Omega 
 +( \wbf^{n_k}, \Ebf^{n_k})_\Omega \\&
 = (\vbf^{n_k}, \nabla\times \Ebf^k)_\Omega +
 (\n\times \vbf^{n_k}, \Ebf^{n_k}_\Tan)_{\partial\Omega}
 +(\wbf^{n_k},\Ebf^{n_k})_\Omega .
 \end{split}
\]  
Since $L^2$-inner products of strongly and weakly convergent sequences are convergent, we infer that 
\[
  (\ve_{n_k} \Ebf^{n_k},\Ebf^{n_k})_\Omega  \to 
 (\vbf^\infty, \nabla\times \Ebf^\infty)_\Omega +
 (\n\times \vbf^\infty, \Ebf^\infty_\Tan)_\paOm +
 (\wbf^\infty,\Ebf^\infty)_\Om \;.
\]  
One more integration by parts yields \quad
$
 (\ve_{n_k} \Ebf^{n_k},\Ebf^{n_k})_\Omega \to (\ve_\infty \Ebf^\infty, \Ebf^\infty)_\Omega \ \text{ as $k \to \infty$}\;.
$ 
The same arguments for the subsequence $\{\Hbf^{n_k}\}_{k\in \NN}$ prove that 
$ (\mu_{n_k} \Hbf^{n_k},\Hbf^{n_k})_\Omega \to (\mu_\infty \Hbf^\infty,\Hbf^\infty)_\Omega$. Combining this with the normalizations \eqref{e:EHnorm}, we complete the proof of (C7).
\end{proof}

\begin{proof}[Proof of Theorem \ref{t:c}] 
Assume that $\om_n \to \om_\infty$. In order to prove Theorem \ref{t:c} we continue the line of Lemmata \ref{l:ConvSub1} and \ref{l:ConvSub2} and show that the  $\LL^2_{\ve_\infty,\mu_\infty}$-normalized weak limit $\{\Ebf^\infty,\Hbf^\infty\}$ produced by statements (C1) and (C7) solves 
 \[
     \ii  \nabla \times \Hbf^\infty  = \omega_\infty \ve_\infty \Ebf^\infty, \qquad -\ii  \nabla \times \Ebf^\infty = \omega_\infty \mu_\infty \Hbf^\infty , \qquad \n \times \Ebf^\infty   = Z  \Hbf^\infty_\Tan \ .
 \]
To this end, we use the weak formulation  of Proposition \ref{p:weak} (iii), i.e., we only have to show that 
\begin{equation}\label{e:16inf}
 (\E^\infty, \ii \nabla\times \H_* - \overline\om\ve_\infty \E_*)_\Omega +
 (\H^\infty,-\ii \nabla\times \E_* - \overline \om \mu_\infty \H_*)_\Omega 
 =0 \quad \text{}
\end{equation}
for all $\{\E_*,\H_*\} \in D_*$, where 
$
D_* :=  \{ \{\E,\H\} \in  H_\imp (\curl,\Om)^2 \ : \ \n \times \E   = - Z^*  \H_\Tan \} .
$
(Note that the domains of operators $(M^{\ve,\mu}_{\imp,Z})^*$ do not depend on the choice of $\ve$ and $\mu$ and are equal to $D_*$.)

From Proposition \ref{p:weak}, we know that for all $n \in \NN$ and $\{\E_*,\H_*\} \in D_*$,
\begin{gather}
 0 = (\Ebf^n, \ii \nabla\times \H_*) - (\Ebf^n,  \overline\om_n\ve_n \E_*)_\Omega
 + (\Hbf^n,-\ii \nabla\times \E_*) - (\Hbf^n, \overline \om_n \mu_n \H_*)_\Omega \ .
 \label{e:weakn}
\end{gather}
Passing to subsequences satisfying the properties (C1)-(C7) and, with abuse of notation, indexing them  again with $n \in \NN$, we use the weak convergences of (C1) in order to obtain 
\begin{equation} \label{e:Limit1}
(\Ebf^n, \ii \nabla\times \H_*)  + (\Hbf^n,-\ii \nabla\times \E_*) \quad \to \quad (\Ebf^\infty, \ii \nabla\times \H_*)  + (\Hbf^\infty,-\ii \nabla\times \E_*) .
\end{equation}
In order to take into account the remaining terms in \eqref{e:weakn}, we consider the expression 
\begin{align*}
  (\Hbf^n,  \overline \om_n \mu_n \H_*)_\Omega  + (\Ebf^n,  \overline\om_n \ve_n \E_*)_\Omega 
& = (\Bbf^n,  \overline \om_\infty \H_*)_\Omega  + (\Dbf^n,  \overline\om_\infty \E_*)_\Omega \\
& \quad \qquad \qquad \qquad - (\om_\infty - \om_n) \, (\Bbf^n,   \H_*)_\Omega  
-  (\om_\infty - \om_n) (\Dbf^n,   \E_*)_\Omega,
\end{align*}
which converges as $n \to \infty$ to 
\[
(\Bbf^\infty,  \overline \om_\infty \H_*)_\Omega  + (\Dbf^\infty,  \overline\om_\infty \E_*)_\Omega  = 
 (\Hbf^\infty,  \overline \om_\infty \mu_\infty \H_*)_\Omega  + (\Ebf^\infty,  \overline\om_\infty \ve_\infty \E_*)_\Omega .
\]
The combination of this limit  with \eqref{e:Limit1} and \eqref{e:weakn} implies \eqref{e:16inf} and completes the proof.
\end{proof}

\begin{proof}[Proof of Theorem \ref{t:EigFree}]
We combine now Lemmata \ref{l:ConvSub1}-\ref{l:ConvSub2} with  Theorem \ref{t:c} and Proposition \ref{p:mDRed} in order to prove Theorem \ref{t:EigFree} by  reductio ad absurdum.
We put $\Mc_n :=  M_{\imp,Z}^{\ve_n,\mu_n}$ for $n \in \NN$ and $\Mc_\infty = M^{\ve_\infty,\mu_\infty}_{\imp,Z}$, where $\ve_\infty$ and $\mu_\infty$ are from Theorem \ref{t:c}. Assuming that 
$\om_n \neq 0$ for all $n \in \NN$ and that $\om_n$ converge to $\om_\infty = 0 $ as $n \to \infty$, we will show that this leads to a contradiction. 

Passing to subsequences, one can ensure that conditions (C1)-(C7) of Lemmata \ref{l:ConvSub1}-\ref{l:ConvSub2} are satisfied. We index the corresponding subsequences by $n \in \NN$ for the simplicity of notation.

Note that the kernels 
$
\KK =  \{ \{\E,\H\} \in  H_\imp (\curl,\Om)^2 \ : \ 0 = \nabla \times \Ebf =  \nabla \times \Hbf, \quad \n \times \E   =  Z \H_\Tan \}
$
of operators $M_{\imp,Z}^{\ve,\mu}$ do not depend on $\ve$ and $\mu$, and that $\ker (M_{\imp,Z}^{\ve,\mu})^* = \ker M_{\imp,Z}^{\ve,\mu} = \KK$ due to Proposition \ref{p:mDRed} (ii).
Thus,
$
\KK = \ker \Mc_n = \ker \Mc_n^*  
$
for all $n \in \NN \cup \{\infty\}$.
By Proposition \ref{p:mDRed}, $\KK$ is a reducing subspace for $\Mc_n$ for every $n \in \NN \cup \{\infty\}$. That is, denoting by $X_n$ the $\LL_{\ve_n,\mu_n}^2$-orthogonal complement $X_n = \LL_{\ve_n,\mu_n}^2 (\Om) \ominus \KK$, we obtain that
\begin{gather}  \label{e:Kreduction}
\text{ $\LL_{\ve_n,\mu_n}^2 (\Om)  = X_n  \oplus \KK$ 
reduces $\Mc_n$ to $\Mc_n= \Mc_n |_{X_n} \oplus 0$ for every $n \in \NN \cup \{\infty\}$.}
\end{gather}
Moreover, $\ker (\Mc_n |_{X_n}) = \{0\}$.
Since $\om_n \neq 0$ for each $n \in \NN$,   the corresponding eigenfield $\{\Ebf^n,\Hbf^n\}$ belongs to $ X_n $ for each $n \in \NN$. This and \eqref{e:Kreduction} implies 
$
0 = (\ve_n \Ebf^n , \Ebf^0)_\Om + (\mu_n \Hbf^n , \Hbf^0)_\Om $
for all $\{\Ebf^0,\Hbf^0\} \in \KK$ and all $n \in \NN $. 
Passing to the limit we get 
$
0 = (\ve_\infty \Ebf^\infty , \Ebf^0)_\Om + (\mu_\infty \Hbf^\infty , \Hbf^0)_\Om 
$
for all $\{\Ebf^0,\Hbf^0\} \in \KK$ (we have used (C6) here). 
This implies 
\begin{gather} \label{Kortinfty}
\{ \Ebf^\infty ,  \Hbf^\infty\} \in X_\infty, \text{ \quad where } X_\infty := \LL_{\ve_\infty,\mu_\infty}^2 (\Om) \ominus \KK .
\end{gather}

Since $\om_n \to 0 $, Theorem \ref{t:c} implies that $\{ \Ebf^\infty ,  \Hbf^\infty\}$ is an $\LL^2_{\ve_\infty,\mu_\infty}$-normalized eigenfield of $M_\infty$ corresponding to the eigenvalue $\om_\infty = 0$, and so,  $\{ \Ebf^\infty ,  \Hbf^\infty\}  \in \KK$. This contradicts
\eqref{Kortinfty}.
\end{proof}

\subsection{Proofs of Proposition \ref{p:real} and of Corollaries \ref{c:NoReala} and \ref{c:M00Eig}}
\label{s:real}

\begin{proof}[Proof of Proposition \ref{p:real}]
Let $\om \in \RR$ be an eigenvalue of m-dissipative operator $M_{\imp,Z}$ and let $ \{\Ebf , \Hbf \}$ be a corresponding eigenfield. Then, by Proposition \ref{p:mDRed} (ii), $ \{\Ebf , \Hbf \}$ is also an eigenfield associated with the same eigenvalue $\om$ for the adjoint operator $M_{\imp,Z}^*$. That is, $ \{\Ebf , \Hbf \} \in D (M_{\imp,Z}) \cap D (M_{\imp,Z}^*)$. Proposition \ref{p:weak} (i) implies that  $Z  \Hbf_\Tan  = \n \times \Ebf   = - Z^*  \Hbf_\Tan$, and 
so $0 = (Z+Z^*) \Hbf_\Tan = 2 \Re (Z) \Hbf_\Tan  $. This and  \eqref{e:Zsec} implies $\Hbf_\Tan = 0$, and, in turn, implies also $\n \times \E = 0$. Thus, $\{\Ebf , \Hbf \}$ is an eigenvector of the symmetric Maxwell operator $M_{0,0}$. 
\end{proof}

We prove now the following more general version of Corollary \ref{c:NoReala}.

\begin{corollary} \label{c:NoReal}
Assume that the impedance operator $Z$ satisfies \eqref{e:Z1}, \eqref{e:Z2}, and \eqref{e:Zsec}. 
Assume that the material parameters $\ve$ and $\mu$ are piecewise Lipschitz in the sense that the assumptions \eqref{e:Lip11}-\eqref{e:Lip2} are satisfied.
Then $\RR \cap \si (M_{\imp,Z}) = \{0\}$.
\end{corollary}
\begin{proof}
As it is explained in Section \ref{ss:Disc}, $0$ is an eigenvalue of $M_{\imp,Z}$. 
Assume now that $\om \in \RR \setminus \{0\}$ and $\om \in \si (M_{\imp,Z})$. Theorem \ref{t:DiscSp}
implies that $\om$ is an eigenvalue of $M_{\imp,Z}$. Any associated $\Le$-normalized eigenfield $ \{\Ebf , \Hbf \}$ satisfies $0 = \n \times \Ebf = \Hbf_\Tan$ due to Proposition \ref{p:real}. We extend $\Ebf $ and $\Hbf$ by $0$ to $\RR^3 \setminus \Om$ and extend $\ve$ and $\mu$ by the constant matrix $I_{\RR^3}$  to $\RR^3 \setminus \Om$. Then the unique continuation result of \cite[Theorem 2.1]{BCT12} implies that $\Ebf  = \Hbf = 0$ a.e. in $\RR^3$. This contradicts  the fact that $ \{\Ebf , \Hbf \}$ is an $\Le$-normalized eigenfield and completes the proof.
\end{proof}

\noindent \emph{Proof of Corollary \ref{c:M00Eig}.}
Let us take an arbitrary $Z$ satisfying \eqref{e:Z1}, \eqref{e:Z2}, and \eqref{e:Zsec}, e.g.,
the identity operator $Z=I_\LLt$. Then, by Proposition \ref{p:real} and Remark \ref{r:realEig}, $\om$ is an eigenvalue of  \eqref{e:EPM00}
if and only if $\om \in \si (M_{\imp,Z}^{\ve,\mu})$. 
This allows us to prove Corollary \ref{c:M00Eig} in the way similar to Corollaries \ref{c:Opt} and \ref{c:OptSA}. Namely, Corollary \ref{c:conv} (ii) implies that for $\FF = L^\infty (\Om, \MM^\sym_{\al, \beta})^2$, 
\[
\hspace{10em} 0<\om_* = \min \{ \om \in \RR_+ \ : \ \om \in \Si [\FF] \setminus \{0\}\} .   \hspace{10em} \qed
\]

\section{Discussion and additional remarks} 
\label{s:discu}

\subsection{Coupled G-closure problems and H-closed feasible families.}\label{ss:coupled}
For Maxwell equations, the periodic homogenization has been intensively studied (see, e.g., \cite{LS18} and references therein). While applications of the homogenization method to optimization problems requires the non-periodic H-convergence, the description of relevant relaxed H-closed feasible families can be often  reduced to G-closure problems of periodic homogenization, see \cite{C00,R01,A02} and  Remark \ref{r:Gclosure}. In Example \ref{ex:PhC1}, the description of the feasible family $\FF$ via  $G_\theta$-closures for two-phase 'conductivity' problem \cite{MT85,T85,LC86} is possible since magnetic permeabilities $\wh \mu_1$ and $\wh \mu_2$ are equal. If $\wh \mu_1 \neq \wh \mu_2$ and $\ep_1 \neq \ep_2$ (e.g., if the magnetic susceptibility of silicon is not neglected) the description of the H-closure is connected with the problem of optimal \emph{coupled bounds} for periodic homogenization. In the 2-dimensional case, the corresponding  coupled G-closure problem is addressed in \cite{CM95} and \cite[Section 11.3]{C00}. However, we were not able to find in the existing mathematical literature a general solution of the coupled G-closure problem for two materials in the 3-dimensional case.

From the point of view of the problem of high-Q design for optical cavities, it is especially interesting to combine this coupled H-closure problem with the unique continuation problem of Section \ref{ss:uc}. We discuss this in the next remark.


\subsection{Extremal cases of nonunique continuation.}
\ \ As it was discussed in Remark \ref{r:nonUC}, the nonunique continuation results of \cite{D12} imply that for every $\al>0$ there exist $\beta>\al$, $\vep,\mu \in L^\infty (\Om, \MM^\sym_{\al, \beta})$, and $\om>0$ such that   the 'overdetermined' on $\pa \Om$
eigenvalue problem 
 \begin{equation} \label{e:EPM00R}
     \ii  \nabla \times \Hbf  = \omega \ve \Ebf, \qquad -\ii  \nabla \times \Ebf = \omega \mu \Hbf , \qquad 0 =\n \times \Ebf,  \qquad 0  = \Hbf_\Tan
     \end{equation}
possesses a nontrivial solution $\{\Ebf,\Hbf\}$. 

Eigenvalues $\om$ of \eqref{e:EPM00R}
are exactly eigenvalues of the symmetric Maxwell operator $M^{\ve,\mu}_{0,0}$. Their study by classical perturbation methods (see, e.g., \cite{Kato}) is difficult since $\si (M^{\ve,\mu}_{0,0}) = \CC$, and so, these eigenvalues are not isolated (in the sense that they are surrounded by the points of residual spectrum of $M^{\ve,\mu}_{0,0}$). 

The meaning of Remark \ref{r:realEig} is that the set of eigenvalues $\om$ of \eqref{e:EPM00R}
can be realized as the part $\RR \cap \si (M_{\imp,Z}^{\ve,\mu}) $ of the spectrum of an m-dissipative Maxwell operator $M_{\imp,Z}^{\ve,\mu}$ under the assumption that $Z$ satisfies \eqref{e:Z1}, \eqref{e:Z2}, and \eqref{e:Zsec}.
For instance, this statement is valid if we take the identity operator $Z=I$ in $\LLt$. The advantage of such a representation is that now all points of the spectrum of $M_{\imp,Z}^{\ve,\mu}$  are isolated eigenvalues (due to Theorem \ref{t:DiscSp}), and so, the perturbation theory of \cite{Kato} is applicable to the study of eigenvalues of \eqref{e:EPM00R}. 

From this point of view, various extreme cases of eigenvalues of \eqref{e:EPM00R} are especially interesting. Note that Corollary \ref{c:M00Eig} can be extended essentially without changes to any H-closed family $\FF \subseteq L^\infty (\Om, \MM^\sym_{\al, \beta})^2$ that  satisfies   the condition 
that the corresponding set $\Si_{0,0} [\FF] := \bigcup\limits_{\{\ve,\mu\} \in \FF} \si (M_{0,0}^{\ve,\mu})$ of achievable eigenvalues over $\FF$ intersects $\RR\setminus \{0\}$.
Namely, in this case the arguments  of the proof of   Corollary \ref{c:M00Eig} show that 
$ \om^+_* (\FF) = \inf (\RR_+ \cap \Si_{0,0} [\FF])$ is a positive eigenvalue of $M_{0,0}^{\ve_*,\mu_*}$ for a certain $\{\ve_*,\mu_*\} \in \FF$. 
 In a certain sense, the value of  $\om^+_* (\FF)$ quantifies  the nonunique continuation property  in the family $\FF$. If $\Si_{0,0} [\FF] \cap \RR_+ = \varnothing$, it is natural to put $\om^+_* (\FF) = +\infty$.
 
For the problem of high-Q design, the value of $\om^+_* (\FF)$ for the family $\FF$ of Example \ref{ex:PhC1} (i) is especially interesting, as well as the value of 
$\om^+_* (\FF)$ for the coupled H-closures discussed in Section \ref{ss:coupled}.
Note that $\om^+_* (\FF) <+\infty$ means that there exists a structure in $\FF$ with the quality-factor $Q=+\infty$ (at least on the level of the idealized model \eqref{e:EP}, \eqref{e:ZBC}).

Modifying  parameters $\al$ and $\beta$ (or $\ep_1$, $\ep_2$, $\wh \mu_1$, $\wh \mu_2$), one can quantify the nonunique continuation in other ways. The following corollary provides simplest examples of such quantifications. 

\begin{corollary} \label{c:M00abBounds}
Let $\om >0$, $\al_0>0$, and $\beta_0>\al_0$ be such that $\om$ is an eigenvalue of \eqref{e:EPM00R} for a certain pair $\{\vep,\mu\} \in L^\infty (\Om, \MM^\sym_{\al_0, \beta_0})^2$.

\item[(i)] Let us define $\al_* = \al_* (\om,\beta_0,\Om) $ and $\beta_* = \beta_* (\om,\al_0,\Om) $ by
\begin{gather*} 
\al_*  = \sup  \{\ \al \ge \al_0 : \ \om \ \text{ is an eigenvalue of \eqref{e:EPM00R}  for a certain  $\{\ve,\mu\} \in L^\infty (\Om, \MM^\sym_{\al, \beta_0})^2$} \}, \\
\beta_* = \inf  \{\ \beta \le \beta_0 : \ \om \ \text{ is an eigenvalue of \eqref{e:EPM00R}  for a certain $\{\ve,\mu\} \in L^\infty (\Om, \MM^\sym_{\al_0, \beta})^2$} \}. 
\end{gather*}
Then  $\al_* < \beta_0$, $\al_0 <\beta_*$, there exists $\{\ve,\mu\} \in L^\infty (\Om, \MM^\sym_{\al_*, \beta_0})^2$ such that $\om$ is an eigenvalue for \eqref{e:EPM00R},
and there exists $\{\ve,\mu\} \in L^\infty (\Om, \MM^\sym_{\al_0, \beta_*})^2$ such that $\om$ is an eigenvalue for \eqref{e:EPM00R}.

\item[(ii)] Let $\om_{\min} = \om^+_* (L^\infty (\Om, \MM^\sym_{\al_0, \beta_0})^2)$. Then there exists an interval $[\al_1, \beta_1] \subseteq [\al_0,\beta_0]$ extremal in the sense  that simultaneously 
\[
\om_{\min} = \om^+_* (L^\infty (\Om, \MM^\sym_{\al_1, \beta_1})^2), \quad \al_1 = \al_* (\om_{\min}, \beta_1,\Om), \quad
\text{ and } \quad \beta_1 = \beta_* (\om_{\min}, \al_1,\Om). 
\]
\end{corollary}

\begin{proof}
Similarly to Corollaries \ref{c:conv}, \ref{c:OptSA}, and \ref{c:M00Eig},
the proof is obtained from the H-compactness  \eqref{e:M2comp} and Theorem \ref{t:EigFree} by iterative applications of Theorem \ref{t:c} to eigenvalues of $M_{\imp,I}^{\ve,\mu}$. Note that 
 $\al_* < \beta_0$ and $\al_0 <\beta_*$ follow from Corollary \ref{c:NoReala}.
\end{proof}

\subsection{Optimization with excluded zero eigenvalue.} Recall that $\Ic = [\vphi_-,\vphi_+] \subset \RR$ is an arbitrary fixed compact interval and 
$d_\Ic (\om)= \min_{\vphi_- \le \vphi \le \vphi_+} |\om - \vphi|$ is the distance from an eigenvalue $\om$ to $\Ic$.

The zero eigenvalue of Maxwell operators $M_{\imp,Z}^{\ve,\mu}$ has no special significance for some of applications and, in particular, for high-Q resonators. The part (ii) of Corollary \ref{c:conv} is designed to exclude the zero eigenvalue from the $d_\Ic$-minimization  problem of Section \ref{ss:OptRes}. 

However,  one has to take care about the abstract possibility of the case of the empty spectrum for an m-dissipative  operator $T_1 = T|_{X \ominus \ker T}$ that appears after factorization of the kernel
of an m-dissipative operator $T$ in a Hilbert space $X$. (In this case the inverse $T_1^{-1}$ is quasi-nilpotent. Compact dissipative quasi-nilpotent operators have been studied in the theory of abstract Volterra operators, see \cite{SFBK10}.)
The physical intuition says that, in the context of wave equations, such operators are very exceptional.   However, taking into account the example of an m-dissipative operator with empty spectrum for the damped string equation \cite{CZ95}, it is difficult to exclude the possibility that there exist a Lipschitz domain $\Om$, bounds $\al$ and $\beta$ satisfying $0<\al<\beta$,  material parameters $\{\ve,\mu\} \in \Mabsym^2 $, and a certain m-dissipative boundary condition (in the sense of \cite{EK22}) such that the corresponding m-dissipative Maxwell operator $M$ has $\si (M) = \{0\}$. 

Therefore, the exclusion of the zero eigenvalue from the optimization requires the following modification of the existence  result of Corollary \ref{c:Opt}: \emph{if $\FF$ is an H-closed subset of $\Mabsym^2$ satisfying $\Si [\FF] \neq \{0\}$, then there exist a feasible material parameter pair $\{ \ve_* , \mu_* \} \in \FF$ and an achievable eigenvalue $\om_* \in  \si (M_{\imp,Z}^{\ve_*,\mu_*}) \setminus \{0\}$ such that 
$ d_\Ic (\om_*) = \min\limits_{\om \in \Si [\FF] \setminus \{0\}}   d_\Ic (\om)$.} 
  This result can be proved in a way similar to the proof of Corollary \ref{c:Opt} with the additional  use of statement (ii) of Corollary \ref{c:conv}.

\section*{Appendix}
\appendix

\section{M-dissipativity of generalized impedance boundary conditions}\label{a:mD}

 It was proved in \cite{EK22} by the boundary tuple method that,
\begin{gather} \label{e:tmdis}
\text{under assumptions 
 \eqref{e:Z1}-\eqref{e:Z2}, the operator $M_{\imp,Z} $ is m-dissipative in $\Le$.}
\end{gather}
\emph{We give here a more elementary proof of \eqref{e:tmdis} that relies on the Lax-Milgram lemma.}

The integration by parts, the boundary condition $\n   \times \Ebf   = Z  \Hbf_\Tan$, and the accretivity of $Z$ imply 
$    2\Im (M_{\mathrm{imp},Z} \ \Phi,\Phi)_\Le 
   = 
   -2\Re ( \n\times \Ebf, \Hbf_{\mathrm{tan}})_\paOm 
   = -2  \Re (Z  \Hbf_\Tan, \Hbf_{\mathrm{tan}})_\paOm 
   \le 0  
   $  
for all $\Phi=\{\Ebf,\Hbf\}\in D(M_{\mathrm{imp},Z}) $. 
This shows that $M_{\imp,Z}$ is dissipative.

Let $ \om = \ii \kappa $ with a certain $\kappa>0$.
If we show that $\ran(\om - M_{\mathrm{imp},Z}) = \Le$, then \cite[Section 1.1]{P59} implies that $M_{\mathrm{imp},Z}$ is m-dissipative. For that it will suffice to show that the equation
\[
 \om \mat{\Ebf \\ \Hbf} - M_{\mathrm{imp},Z} \mat{ \Ebf\\ \Hbf} = \mat{\fbf^1 \\ \fbf^2}
\]
has a  solution $\Phi=\{\Ebf,\Hbf\}  \in D(M_{\mathrm{imp},Z})$ for all $\fbf = \{\fbf^1,\fbf^2\}\in L^2 (\Om,\CC^6)$. 

We  consider for $\Hbf \in \Himp$  the following variational equation, 
\begin{equation}\label{20}
  b(\H,\wbf)= 
-\frac{1}{\om} ( \fbf^1,\nabla\times \wbf)_\Om  - {\rm i}(\mu  \fbf^2, \wbf)_\Omega 
\quad \mbox{ for all } \wbf\in \Himp,
 \end{equation}
where
$
 b(\ubf,\wbf) =-{\rm i}\om( \mu \ubf,\wbf)_\Omega + 
 \frac{\rm i}{\om}
  (\ve^{-1}\nabla \times \ubf,\nabla\times \wbf)_\Omega 
  +(Z \ubf_{\mathrm{tan}},\wbf_{\mathrm{tan}})_{\partial\Omega}
$  
is a bounded and coercive sesquilinear form on $\Himp$. 
Indeed, the boundedness of $b (\cdot,\cdot)$ is clear from \eqref{e:NormHimp}.
The coercivity  of $(Z\cdot,\cdot)_\paOm$ in $\LLt$ implies (see, e.g., \cite[Remark 4.2.7]{ACL18}) 
that there exists $\vth \in \RR$ such that $\Re \left( \ee^{\ii \vth} (Z\vbf,\vbf)_\LLt \right) \gtrsim \|\vbf\|_\paOm^2$ for all $\vbf \in \LLt$. 
From the accretivity of $Z$, one obtains that the corresponding rotation angle  $\vth$ can be chosen in the interval $(-\de,\de)$ with arbitrary small $\de \in (0, \pi/4]$. This yields 
\[
 |b(\ubf,\ubf)| \ge \cos (\pi/4) \left(   
\kappa ( \mu \ubf,\ubf)_\Omega + 
 \frac1{\kappa} (\ve^{-1}\nabla \times \ubf,\nabla\times \ubf)_\Omega  \right)
  + |( Z \ubf_{\mathrm{tan}},\ubf_{\mathrm{tan}})_{\partial\Omega}|
  \gtrsim  \|\ubf\|_{\Himp}^2 
\] 
for all $\ubf\in \Himp$, 
which proves that $b$ is coercive.

The Lax-Milgram lemma implies that for every $\fbf = \{\fbf^1,\fbf^2\}\in L^2 (\Om,\CC^6)$ there exists a unique solution $\Hbf \in \Himp$ to \eqref{20}.
Now, we define $\Ebf = \om^{-1} \left(\fbf_1+  \ii \ve^{-1} \nabla\times \Hbf\right)$. 
Then $\Ebf \in L^2(\Omega,\C^3)$ and, for all $\wbf \in \Himp$,  
\begin{equation}\label{5}
  - \ii \om( \mu \Hbf, \wbf)_\Omega + 
  (\Ebf,\nabla\times \wbf)_\Omega 
  +(Z \Hbf_\Tan, \wbf_\Tan)_{\partial\Omega}
  =- {\rm i}(\mu  \fbf_2, \wbf)_\Omega \ .
\end{equation}
Choosing $\wbf \in H_0(\curlm,\Omega)$, one sees that  $\nabla \times \Ebf \in L^2(\Omega,\C^3)$ and $\om \mu \Hbf + {\rm i} \nabla\times \Ebf = \mu \fbf_2$. Hence, $\Phi = \{\Ebf, \Hbf\}$ solves the two Maxwell equations under consideration. 

It remains to prove that $\Phi = \{\Ebf, \Hbf\}$ satisfies the boundary condition $\n   \times \Ebf   = Z  \Hbf_\Tan$. In equation \eqref{5} we perform the integration by parts in the second term and make use of the second Maxwell equation. Then, for all $\wbf \in \Himp$,
\begin{gather}
  - \< \n \times \Ebf , \wbf_\Tan \>_{\partial\Omega}+ 
  (Z \Hbf_\Tan , \wbf_\Tan )_{\partial\Omega}
  =0, \label{e:yOnPaOm}
\end{gather}
where $  \< \cdot , \cdot  \>_{\partial\Omega}$ is the pairing between $H^{-1/2}(\Div_{\partial\Omega},\pa \Om)$ and $H^{-1/2}(\curl_{\partial\Omega}, \pa \Om)$ w.r.t. the pivot space  $\LLt$. For basic facts concerning this duality, and in particular, for the equality 
\begin{gather} \label{e:wTan}
\{\wbf_\Tan \, : \, \wbf \in \Himp\} = \LLt \cap H^{-1/2}(\curl_\paOm ,\pa \Om), 
\end{gather}
we refer to  \cite{BCS02,BHPS03,EK22}.
Since $\LLt \cap H^{-1/2}(\Div_{\partial\Omega},\pa \Om)$ is dense in $ \LLt$ (see Section \ref{ss:Hodge}), it follows from \eqref{e:yOnPaOm}  and \eqref{e:wTan} that  $\n \times \Ebf \in \LLt$ and 
$\n \times \Ebf - Z \Hbf_\Tan=0$ in the sense of $\LLt$. This completes the proof of \eqref{e:tmdis}.

\section{Proof of Proposition \ref{p:Hcon}, equivalence of two H-convergences}\label{s:EquivHconv}

We were not able to reach the publication of  Tartar \cite{T85}, where Proposition \ref{p:Hcon} is originated from.
However, in our opinion, the proof given below should be very close to the original proof of \cite{T85}  since it is heavily based on the methods of compensated compactness and oscillating  test functions of Murat \& Tartar \cite{MT78} (see also \cite{A02,T09}).

\begin{proof}[Proof of Proposition \ref{p:Hcon}.] Recall that $\<\cdot,\cdot \> $ denotes the standard inner product in $\CC^3$ or $\RR^3$. By $\one $ we denote the constant function equal to $1$. The proof is split in several steps.
 
\emph{Step 1. The construction of test functions. }
 Suppose that $\{A_n\}_{n \in \NN} \subseteq \Mab$ H-converges to $A_\infty$. For any $y \in \RR^3$, there exists a sequence of functions $\{w_n\}_{n \in \NN} $ such that $w_n  \Weak   \langle y,x\rangle \one$  in  $H^1(\Om)$, $\nabla w_n \Weak \one y  $ in $L^2(\Omega,\C^3)$, $A_n \nabla w_n \rightharpoonup A_\infty \one y $ in $L^2(\Omega,\C^3)$, and
$ \nabla\cdot (A_n\nabla w_n)  = g $
  for a certain $g \in H^{-1}(\Omega)$, where $\one y$ stands for the constant $\RR^3$-valued function equal to $y$.
 The existence of such a sequence $\{w_n\}_{n \in \NN} $ is a part of the oscillating test
function method of  \cite{MT78}  (see \cite[Section 1.3.1]{A02}).



\emph{Step 2. Proof of 'only if'.} Suppose that $\{A_n\}_{n \in \NN} \subseteq \Mab$ H-converges to $A_\infty$.
Suppose that  $\Ebf^n \in L^2 (\Om,\RR^3)$ and  $\Dbf^n = A_n \Ebf^n \in L^2 (\Om,\RR^3)$ satisfy the hypotheses \eqref{e:DE1}-\eqref{e:DE4} for all $n\in \NN$. 
For each $n \in \NN \cup \{\infty\}$, we consider now  certain  measurable defined a.e. functions $\Ebf^n$, $\Dbf^n$, $A_n$, and $w_n$ representing the corresponding equivalence classes $\Ebf^n, \Dbf^n \in L^2 (\Om,\RR^3)$,  $A_n \in \Mab$, and, resp., $w_n \in H^1 (\Om)$.
Since a countable union of sets of measure zero is also of zero measure, there exists a set $\Om_0$ of measure zero such that, for all $\NN$, the functions $\Ebf^n$, $\Dbf^n$, $A_n$,  $w_n $, and $A_\infty$ are defined 
on $x \in \Om \setminus \Om_0$ and $\Dbf^n = A_n \Ebf^n$ for all $x \in  \Om \setminus \Om_0$.

Evaluating $\langle \Dbf^n  - A_n \nabla w_n, \Ebf^n - \nabla w_n \rangle$  pointwise in $\Om$, one sees that for a.a. $x \in \Omega \setminus \Om_0$,
\begin{gather*}
 \langle \Dbf^n  - A_n \nabla w_n, \Ebf^n - \nabla w_n \rangle =
 \langle A_n (\Ebf^n-\nabla w_n), \Ebf^n - \nabla w_n \rangle \ge \alpha |\Ebf^n - \nabla w_n|^2\ge 0\;. 
\end{gather*}
Observing  that $\nabla \times (\Ebf^n-\nabla w_n) = \nabla \times \Ebf^n$ and $\nabla\cdot A_n (\Ebf^n-\nabla w_n) = \nabla\cdot \Dbf^n - g$, we infer from \eqref{e:DE3}-\eqref{e:DE4} that $\{\nabla \times (\Ebf^n-\nabla w_n) \}_{n \in \NN}$ is a relatively compact subset of $H^{-1}(\Omega,\C^3)$ and that $\{\nabla\cdot A_n (\Ebf^n-\nabla w_n) \}_{n \in \NN}$ is a relatively compact subset of $H^{-1}(\Omega)$. Hence, using  \cite[Lemma 1.3.1]{A02} (which is a version of the div-curl lemma), we infer that,  after a possible replacement of $\Om_0$ with a larger  measure zero set $\Om_0$, there is a dense countable subset $Y$ of $\RR^3$ such that \linebreak
$
 \langle \Dbf^\infty - A_\infty y, \Ebf^\infty - y \rangle \ge 0 $  for all $x \in \Om \setminus \Om_0$ and all $y \in Y$.
Using now the density of $Y$ in $\RR^3$, one obtains   
$  \langle \Dbf^\infty - A_\infty y, \Ebf^\infty - y \rangle \ge 0 $ 
 for all $x \in \Om \setminus \Om_0$ and all  $ y \in \RR^3$.

Let $\underline{x}\in \Omega \setminus \Om_0$. 
Then, taking $y = \Ebf^\infty(\underline{x}) - tz$ with arbitrary $z \in \R^3$ and  $t>0$, we obtain  
\[
 \langle \Dbf^\infty(\underline{x}) - 
 A_\infty(\underline{x}) \Ebf^\infty(\underline{x})+
 t A_\infty(\underline{x}) z , z \rangle\ge 0 \quad \text{ for all $z \in \R^3$ and  all $t>0$}.
\] 
Letting $t\to 0^+$ proves that $   \langle \Dbf^\infty(\underline{x}) - 
  A_\infty(\underline{x}) \Ebf^\infty(\underline{x}) , z \rangle\ge 0 $ \  for all $\underline{x} \in \Om \setminus \Om_0$ and all $ z \in \RR^3$.
Thus,  the equality $\Dbf^\infty =A_\infty \Ebf^\infty$ holds a.e. in $\Om$.



\emph{Step 3. Proof of `if'.}
Suppose that $\Dbf^\infty =A_\infty \Ebf^\infty$ a.e. in $\Om$ for all weak $L^2$-limits $\{\Ebf^\infty, \Dbf^\infty\}$ of sequences $\{\Ebf^n, \Dbf^n\}$ satisfying \eqref{e:DE1}-\eqref{e:DE4}.
 
Given $f\in H^{-1}(\Omega)$, we consider for each $n \in \NN$ the weak solution to the Dirichlet  problem 
\[
 \nabla\cdot (A_n \nabla u_n) = f \mbox{ in } \Omega \ , \qquad \quad u_n =0 \mbox{ on } \partial\Omega\;.
\]
It can be seen from the Lax-Milgram lemma that the sequence $\{u_n\}_{n \in \NN}$ is bounded in $H^1_0(\Omega)$.  Hence, there exists a weakly convergent subsequence (for brevity also indexed by $n$) such that $u_n \Weak u_\infty$ in $ H^1_0(\Omega)$ for a certain $u_\infty \in H^1_0(\Om)$. Set $\Ebf^n =\nabla u_n$. Then $\Ebf^n \rightharpoonup \Ebf^\infty$ in $L^2(\Omega,\C^3)$ with $\Ebf^\infty =\nabla u_\infty$, and $\nabla \times \Ebf^n =0$ for all $n\in \NN$. We put now  $\Dbf^n=A_n \Ebf^n$, and observe that $\nabla\cdot \Dbf^n = f \in H^{-1}(\Omega)$ and that the sequence $\{\Dbf^n\}_{n \in \NN}$ is bounded in $L^2(\Omega,\CC^3)$. Hence, the sequence $\{\Dbf^n\}_{n \in \NN}$ has 
a subsequence with a certain weak $L^2$-limit $\wlim_{n \to \infty} \Dbf^n = \Dbf^\infty$. 

We see that this construction of sequences $\{\Ebf^n\}_{n \in \NN}$ and $\{\Dbf^n\}_{n \in \NN}$ ensures that they satisfy  assumptions \eqref{e:DE1}-\eqref{e:DE4}. Thus,  $\Dbf^\infty = A_\infty \Ebf^\infty = A_\infty \nabla u_\infty$ a.e. in $\Om$ and 
\[
\wlim A_n \nabla u_n = \wlim \Dbf^n  =  \Dbf^\infty = A_\infty \nabla u_\infty \quad \text{in the sense of $L^2 (\Om,\CC^3)$.}
\]

Furthermore, for arbitrary $\psi \in H^1_0(\Omega)$,
\[
\<f,\psi\>_\Om = \<\nabla\cdot \Dbf^n, \psi\>_\Om = -  (\Dbf^n, \nabla \psi)_\Om = - \lim_{n \to \infty}  (\Dbf^n, \nabla \psi)_\Om = - (\Dbf^\infty,\nabla \psi)_\Om = \<\nabla\cdot \Dbf^\infty,\psi\>_\Om \;,
\]
where $\<\cdot,\cdot\>_\Omega$ is  the pairing of $H^{-1}(\Omega)$ and $H^1_0(\Omega)$. This implies 
$
\nabla\cdot (A_\infty \nabla u_\infty) = \nabla \cdot \Dbf^\infty =f.
$

Thus, the sequence  $\{A_n\}_{n\in \NN}$ H-converges to $A_\infty$. This completes the proof.
\end{proof}

\section{Proof of Theorem \ref{t:HDecOm}, Helmholtz-Hodge decompositions in $\Om$}
\label{a:HdecOm}

Let $\xi \in \Mabsym$ and  $\ubf \in L^2_\xi (\Om,\CC^3)$. 
Let $\Ga^1, \Ga^2, \dots, \Ga^N$ be the connected components of $\partial\Omega$ 
(note that $N \in \NN$ since $\Om$ is a Lipschitz domain). 
Let $\<\cdot,\cdot \>_{L^2 (\pa \Om)}$ be the sesquilinear paring of $H^{-1/2} (\pa \Om)$ and $H^{1/2} (\pa \Om)$ w.r.t. the pivot space  $L^2 (\pa \Om)$ of  $L^2$-scalar-fields on $\pa \Om$ (we use similarly the notation $\<\cdot,\cdot \>_{L^2 (\Ga^j)}$).

Let $\grad_0:H^1_0 (\Om) \to L_\xi^2 (\Omega,\CC^3)$ be the map $u \mapsto  \nabla u $ understood as a bounded operator from $H^1_0 (\Om)$ to $L^2_\xi (\Omega,\CC^3)$. The Poincaré inequality implies that $\ran \grad_0 = \grad H^1_0 (\Om)$  is a closed subspace in $ L_\xi^2 (\Omega,\CC^3)$ and that 
\begin{gather} \label{e:grad0homeom}
\text{$\grad_0$ considered as an operator from $H^1_0 (\Om)$  to $\grad H^1_0 (\Omega) $ becomes a homeomorphism.}
\end{gather}
The integration by parts implies that
the orthogonal complement $L_\xi^2 (\Omega,\CC^3) \ominus \grad H^1_0 (\Om)$ is the closed subspace $H(\Div \xi 0,\Om)$ defined by \eqref{e:HdivXi0}. In the case $\xi \equiv I_{\RR^3}$, we denote this space $H(\Div  0,\Om)$ and equip it with the $L^2 (\Om,\CC^3)$-norm. Recall that $\curlm_1:H^1 (\Om,\CC^3) \to L^2 (\Omega,\CC^3)$ is the bounded operator defined in Section \ref{ss:Hodge}.

\begin{proof}[Proof of Theorem \ref{t:HDecOm}.] 
\emph{Let us prove the statement (i).}
In order to obtain the `weighted' Helmholtz-Hodge decomposition of Theorem \ref{t:HDecOm},  consider the bounded operator $\xi^{-1} \curlm_1:H^1 (\Om,\CC^3) \to L_\xi^2 (\Omega,\CC^3)$ defined by the differential operation 
$  \ubf (\cdot) \mapsto (\xi (\cdot))^{-1} (\nabla \times  \ubf) (\cdot)$, and consider  its image  
$\ran (\xi^{-1} \curlm_1) = \xi^{-1} \curlm H^1 (\Om,\CC^3)$.
The distributional equality $\nabla \cdot (\nabla \times  \ubf) = 0$ implies that 
$
\xi^{-1} \curlm H^1 (\Om,\CC^3) \subseteq H(\Div \xi 0,\Om) .
$ 
The result on the existence of vector potentials of  \cite[Theorem I.3.4]{GR86} (see also \cite{M03}) represents 
$\curlm H^1 (\Om,\CC^3) = \{ \curlm \ubf  \ : \ubf \in  H^1 (\Om,\CC^3)\}$ as 
\begin{gather} \label{e:curlH1}
\curlm H^1 (\Om,\CC^3) = \{ \wbf \in H(\Div  0,\Om) \ : \ \<\ga_\nr (\wbf) , \one \>_{L^2 (\Ga_j)} = 0, \quad 1 \le j \le N\},
\end{gather}
where the bounded operator $\ga_\nr  : H(\Div  0,\Om) \to H^{-1/2} (\pa \Om)$ is the  \emph{normal-component trace} obtained as the extension by continuity   of $\wbf \mapsto \n \cdot \wbf |_\paOm$  from $H^1 (\Om,\CC^3) \cap H(\Div  0,\Om)$  (this can be done with the help of the integration by parts for the $\Div$-operator).
Note that the integration by parts gives for every $\wbf \in H(\Div  0,\Om)$
\begin{gather*} 
0 = (\wbf, \nabla \one)_\Om =  \<\ga_\nr (\wbf) , \one \>_{L^2 (\pa \Om)} - (\nabla \cdot \wbf, \one)_{L^2 (\Om)} = \<\ga_\nr (\wbf) , \one \>_{L^2 (\pa \Om)} .
\end{gather*}

From this equality  and the fact that $\wbf \mapsto \<\ga_\nr (\wbf) , \one \>_{L^2 (\Ga_j)}$ are continuous functionals on $H(\Div  0,\Om)$,
we see that 
$\xi^{-1} \curlm H^1 (\Om,\CC^3)$ is a closed subspace of $H(\Div \xi 0,\Om) $
 of co-dimension $N_1 \le N-1$. 
So the $L^2_\xi$-orthogonal complement 
$ \wt \HH_2 ( \Om,\xi) := H(\Div \xi 0,\Om) \ominus \xi^{-1} \curlm H^1 (\Om,\CC^3)$
is a space of the dimension $N_1 \le N-1$.

Let us show that $\wt \HH_2 ( \Om,\xi)$ is $(N-1)$-dimensional and that $\wt \HH_2 ( \Om,\xi)$ is equal to 
\[  \HH_2 ( \Om,\xi) := 
 \{ \wbf = \nabla q \; : \ q\in H^1(\Om), \ 
 \nabla\cdot (\xi \nabla q)=0 \mbox{ in } \Om, \text{ and }
 q |_\paOm \in \KK_0 (\pa \Om) \} ,
\]
where $\KK_0 (\pa \Om)$ is the $N$-dimensional space of locally constant $\CC$-valued functions  on $\pa \Om$.
Indeed, $ \HH_2 ( \Om,\xi) \subseteq H(\Div \xi 0,\Om) $.
The orthogonalities $\<\ga_\nr (\wbf) , \one \>_{L^2 (\Ga_j)} =0$ in the representation \eqref{e:curlH1} show that, for   $q |_\paOm \in \KK_0 (\pa \Om)$ and $\vbf \in H^1 (\Om,\CC^3)$, we have 
\[
(\xi^{-1} \curlm \vbf ,  \nabla q )_{L^2_\xi (\Om,\CC^3)} = ( \curlm \vbf ,  \nabla q )_\Om = \<\ga_\nr (\curlm \vbf) \,  , \, q|_\paOm \, \>_{L^2 (\pa \Om)} - \<  \nabla \cdot (\nabla \times \vbf ) , q \>_\Om = 0 .
\]
That is, $ \HH_2 ( \Om,\xi)$ and $\xi^{-1} \curlm H^1 (\Om,\CC^3)$ are $L^2_\xi$-orthogonal. 
Since  $ \HH_2 ( \Om,\xi)$ is $(N-1)$-dimensional, we get $\HH_2 ( \Om,\xi) = \wt \HH_2 ( \Om,\xi)$. This finishes the proof of statement (i) of Theorem \ref{t:HDecOm}.

\emph{Let us prove the statement (ii) of Theorem \ref{t:HDecOm}.}
The kernel  $ \ker \curlm_1:= \{\vbf \in H^1 (\Om,\CC^3) : \nabla \times  \vbf = 0 \} $  is a closed subspace of $H^1 (\Om,\CC^3)$. Let us consider  the orthogonal complement $ (\ker \curlm_1)^\perp = H^1 (\Om,\CC^3) \ominus \ker \curlm_1$ and the bounded bijective operator  $\xi^{-1} \curlm_{\perp}:(\ker \curlm_1)^\perp  \to \xi^{-1} \curlm H^1 (\Om,\CC^3)$ defined as the restriction of $\xi^{-1} \curlm_1$ to the subspace $(\ker \curlm_1)^\perp $.  By statement (i) of Theorem \ref{t:HDecOm},  $ \xi^{-1} \curlm H^1 (\Om,\CC^3)$ is a closed subspace of $ L^2_\xi (\Om,\CC^3)$. The bounded inverse theorem implies that $\xi^{-1} \curlm_{\perp}:(\ker \curlm_1)^\perp  \to \xi^{-1} \curlm H^1 (\Om,\CC^3)$ is a homeomorphism. Using \eqref{e:grad0homeom} and the fact that
$\HH_2 ( \Om,\xi) $ is finite dimensional, one obtains  statement (ii) of Theorem \ref{t:HDecOm}.
\end{proof}


\small
\begin{thebibliography}{99}


\bibitem{A02} Allaire, G., Shape Optimization by the Homogenization Method. Springer, 2002.


\bibitem{ATL89} Alvino, A., Trombetti, G. and Lions, P.L., 1989. On optimization problems with prescribed rearrangements. Nonlinear Analysis: Theory, Methods \& Applications, 13(2), pp.185-220.


\bibitem{ACL18} Assous, F.,  Ciarlet, P.  and Labrunie, S. 
Mathematical foundations of computational electromagnetism,
Springer, 2018.

\bibitem{BCT12} Ball, J.M., Capdeboscq, Y. and Tsering-Xiao, B., 2012. On uniqueness for time harmonic anisotropic Maxwell's equations with piecewise regular coefficients. Mathematical Models and Methods in Applied Sciences, 22(11), p.1250036.



\bibitem{BCS02} Buffa, A., Costabel, M. and Sheen, D., 2002. On traces for $H (\curl, \Om)$ in Lipschitz domains. Journal of Mathematical Analysis and Applications, 276(2), pp.845-867.

\bibitem{BHPS03} Buffa, A., Hiptmair, R., Petersdorff, T.V. and Schwab, C., 2003. Boundary element methods for Maxwell transmission problems in Lipschitz domains. Numerische Mathematik, 95, pp.459-485.

\bibitem{C96} Cessenat, M.  Mathematical methods in electromagnetism: linear theory and applications, World Scientific Publishing, Singapore, 1996.


\bibitem{C00} Cherkaev, A., Variational methods for structural optimization. Springer Science \& Business Media, New York, 2000.

\bibitem{CK97} Cherkaev, A. and Kohn, R. Eds., Topics in the mathematical modelling of composite materials. Boston: Birkhäuser, 1997. 


\bibitem{CM95} Clark, K.E. and Milton, G.W., 1995. Optimal bounds correlating electric, magnetic and thermal properties of two-phase, two-dimensional composites. Proc. R. Soc. London A, 448, pp.161-190.


\bibitem{CL96} Cox, S., and Lipton, R., 1996. Extremal eigenvalue problems for two-phase conductors. Arch. Rational Mech. Anal., 136(2), pp.101-118.

\bibitem{CZ95}
Cox, S. and Zuazua, E., 1995. The rate at which energy decays in a string damped at one end. Indiana University Mathematics Journal, 44(2), pp.545-573.


\bibitem{D93} Dal Maso, G.,  An introduction to $\Ga$-convergence. Springer Science \& Business Media, 1993.

\bibitem{D12} Demchenko, M.N.,  2012. Nonunique continuation for the Maxwell system. J. Math. Sci. (N.Y.), 185(4), pp.554-566 (translated from Zapiski Nauchnykh Seminarov POMI 393 (2011), pp.80-100).



\bibitem{EK21}
Eller, M. and Karabash, I.M.,  Euler–Lagrange equations for full topology optimization of the Q-factor in leaky cavities. In ``2021 Days on Diffraction'', Eds. Motygin, O.V., Kiselev, A.P., Goray, L.I., Kirpichnikova, A. S., pp. 29-35, IEEE, 2021.

\bibitem{EK22} Eller, M. and Karabash, I.M., 2022. M-dissipative boundary conditions and boundary tuples for Maxwell operators. Journal of Differential Equations, 325, pp.82-118.

\bibitem{EY06} Eller, M.M. and Yamamoto, M., 2006. A Carleman inequality for the
stationary anisotropic Maxwell system. Journal de Mathématiques Pures et Appliquées, 86(6), pp.449-462.


\bibitem{E12} Exner, P. 
Open quantum systems and Feynman integrals,
Springer Science \& Business Media, Berlin, 2012.

\bibitem{EL21} Exner, P. and Lotoreichik, V., 2021. Optimization of the lowest eigenvalue of a soft quantum ring. Letters in Mathematical Physics, 111(2), p.28.


\bibitem{GMMM11} 
Gesztesy, F., Mitrea, I., Mitrea, D. and Mitrea, M., 2011. On the nature of the Laplace–Beltrami operator on Lipschitz manifolds. J. Math. Sci. (N.Y.), 172(3), pp.279-346.

\bibitem{GR86} Girault V. and Raviart, P.-A., Finite Element Methods for the Navier-Stokes Equations,  Springer, Berlin, 1986.


\bibitem{GHL21} Girouard, A., Henrot, A. and Lagacé, J., 2021. From Steklov to Neumann via homogenisation. Arch. Rational Mech. Anal., 239, pp.981-1023.

\bibitem{GG91} Gorbachuk, V.I. and  Gorbachuk, M.L.   Boundary value problems for operator differential equations, Kluwer Academic Publishers, 1991.


\bibitem{H13} Haroche, S., 2013. Nobel Lecture: Controlling photons in a box and exploring the quantum to classical boundary. Reviews of Modern Physics, 85(3), p.1083.


\bibitem{HS86} Harrell, E.M. and Svirsky, R., 1986. Potentials producing maximally sharp resonances. Transactions of the American Mathematical Society, 293(2), pp.723-736.


\bibitem{HBKW08} Heider, P., Berebichez, D., Kohn, R.V. and Weinstein, M.I., 2008. Optimization of scattering resonances. Structural and Multidisciplinary Optimization, 36(5), pp.443-456.

\bibitem{H06} Henrot, A., 2006. Extremum problems for eigenvalues of elliptic operators. Springer Science \& Business Media.

\bibitem{H17} Henrot, A., Ed. Shape optimization and spectral theory.  De Gruyter Open Poland, Warsaw, 2017.

\bibitem{KS08}
 Kao, C.-Y., Santosa, F., 2008
 Maximization of the quality factor of an optical resonator.
 Wave Motion 
 45, pp.412-427.

\bibitem{KCK20} Kang, D., Choi, P. and Kao, C.-Y., 2020. Minimization of the first nonzero eigenvalue problem for two-phase conductors with Neumann boundary conditions. SIAM Journal on Applied Mathematics, 80(4), pp.1607-1628.


\bibitem{K13} Karabash, I.M., 2013. Optimization of quasi-normal eigenvalues for 1-D wave equations in inhomogeneous media; description of optimal structures. Asymptotic Analysis, 81(3-4), pp.273-295.

\bibitem{K14} Karabash, I.M., 2014. Pareto optimal structures producing resonances of minimal decay under L1-type constraints. Journal of Differential Equations, 257(2), pp.374-414.



\bibitem{KKV20} Karabash, I.M., Koch, H., and Verbytskyi, I.V., 2020. Pareto optimization of resonances and minimum-time control. Journal de Mathématiques Pures et Appliquées, 136, pp.313-355.



\bibitem{Kato} Kato,  T.
Perturbation theory for linear operators.
Springer, New York, 1966.

\bibitem{KS86} Kohn, R.V. and Strang, G., 1986. Optimal design and relaxation of variational problems, II. Communications on Pure and Applied Mathematics, 39(2), pp.139-182.


\bibitem{KK69} Krein, S.G., Kulikov, I.M., 1969. The Maxwell-Leontovich operator, Differentsialnie Uravneniya,  5(7), pp.1275-1282 (in Russian).


\bibitem{LL04} Lagnese, J. and Leugering, G. 
Domain decomposition methods in optimal control of partial differential equations, 
Birkh\"auser-Verlag, Basel, 2004.

\bibitem{LZ21} Lamberti, P.D. and  Zaccaron, 2021. M. Shape sensitivity analysis for electromagnetic cavities. Math. Methods Appl. Sci., 44 (13), pp.10477–10500. 


\bibitem{LL84} Landau, L.D., Lifshitz, E.M., and Pitaevskii, L.P. Electrodynamics of continuous media, Pergamon Press, Oxford, New York, 1984.


\bibitem{LS18} Lipton, R. and Schweizer, B., 2018. Effective Maxwell’s equations for perfectly conducting split ring resonators. Arch. Rational Mech. Anal., 229, pp.1197-1221.


\bibitem{LSV03} 
Lipton, R.P., Shipman, S.P. and Venakides, S., Optimization of resonances in photonic crystal slabs. In Physics, Theory, and Applications of Periodic Structures in Optics II (Vol. 5184, pp. 168-177). SPIE, 2003.


\bibitem{LC86} Lurie, K.A. and Cherkaev, A.V., 1986. Exact estimates of the conductivity of a binary mixture of isotropic materials. Proc. Roy. Soc. Edinburgh Sect. A, 104(1-2), pp.21-38.




\bibitem{M23} Milton, G.W. The theory of composites. Society for Industrial and Applied Mathematics,  2023. 

\bibitem{M04} Mitrea, M., 2004. Sharp Hodge decompositions, Maxwell's equations, and vector Poisson problems on nonsmooth, three-dimensional Riemannian manifolds. Duke Math. J., 125(3), pp.467-547.

\bibitem{M03} Monk, P.
Finite element methods for Maxwell's equations,
Oxford University Press, 2003.


\bibitem{MT78} Murat, F. and Tartar, L., H-convergence, 1978. Lecture notes; English transl. in \cite{CK97}.

\bibitem{MT85}
Murat, F. and Tartar, L., 1985. Calcul des Variations et Homogeneisation,
Les Methodes de l'Homogeneisation Theorie et Applieations en
Physique, Coll. Dir. Etudes et Reeherehes EDF, 57, Eyrolles, Paris,
pp.319-369;  English transl. ``Calculus of variations and homogenization'' in \cite{CK97}.

\bibitem{NW12} Nguyen, T. and Wang, J.N., 2012. Quantitative uniqueness estimate
for the Maxwell system with Lipschitz anisotropic media. Proceedings
of the American Mathematical Society, 140(2), pp.595-605.


\bibitem{OSY92} Oleinik, O.A., Shamaev, A.S. and Yosifian, G.A. Mathematical problems in elasticity and homogenization. Elsevier, 1992.

\bibitem{OW13} Osting, B. and Weinstein, M.I., 2013. Long-lived scattering resonances and Bragg structures. SIAM Journal on Applied Mathematics, 73(2), pp.827-852.

\bibitem{PS22} Pauly, D. and Skrepek, N., 2022. A compactness result for the div-curl system with inhomogeneous mixed boundary conditions for bounded Lipschitz domains and some applications. Annali dell' Università di Ferrara, 69(2), pp.505-519.


\bibitem{P59} Phillips, R.S., 1959. Dissipative operators and hyperbolic systems of partial differential equations. Transactions of the American Mathematical Society, 90(2), pp.193-254.


\bibitem{R01} Raitums, U., 2001. On the local representation of G-closure. Arch. Rational Mech. Anal., 158(3), pp.213-234.


\bibitem{RSIV} Reed, M. and Simon, B.  Methods of modern mathematical physics.  IV: Analysis of operators. Academic Press, New York, 1978.

     
\bibitem{SFBK10} Sz.-Nagy, B., Foias, C., Bercovici, H.,  and Kérchy, L. Harmonic analysis of operators on Hilbert space. Springer, New York, NY, 2010. 

\bibitem{T85} Tartar, L., Estimations fines des coefficients homogénéisés. In Ennio De Giorgi colloquium (Paris, 1983) (Vol. 125, pp. 168-187).  Pitman,  MA: Boston, 1985. 

\bibitem{T09} Tartar, L.,  The general theory of homogenization: a personalized introduction. Springer Science \& Business Media, 2009.

\bibitem{VS21} Vasco, J.P. and Savona, V., 2021. Global optimization of an encapsulated Si/SiO 2 L3 cavity with a 43 million quality factor. Scientific Reports, 11(1), p.10121.


\bibitem{YI18} Yuferev, S.V. and Ida, N.  Surface impedance boundary conditions: a comprehensive approach, CRC press, 2018. 

\end{thebibliography}

\vspace{1ex}
\noindent
\emph{Acknowledgments.} I.~Karabash was supported by the Heisenberg Programme (project 509859995) of the German Science Foundation (DFG, Deutsche Forschungsgemeinschaft), by the CRC 1060 `The mathematics of emergent effects' at the University of Bonn, which is funded through DFG, and by the Hausdorff Center for Mathematics funded by DFG  under Germany's Excellence Strategy - EXC-2047/1 - 390685813. 
IK is grateful to the DFG-funded Hausdorff Research Institute for Mathematics for the possibility to participate in the Trimester Program "Mathematics for Complex Materials'' (3.01--14.04.2023) and to the VolkswagenStiftung project   "From Modeling and Analysis to Approximation" for the financial support of his participation in the workshop ''From Modeling and Analysis to Approximation and Fast Algorithms'' (2-6.09.2023), where results related to the topics of the present paper were discussed with other participants.

\vspace{1ex}

\begingroup
\renewcommand{\section}[2]{}%

\endgroup
     
\end{document}